%% file: HeLiuMok20240123_26-1-2024.tex
\documentclass[11pt]{amsart}
\usepackage{amsmath, amssymb, amscd, mathrsfs, slashed, url}
\usepackage{mathtools}
\usepackage{color}
\usepackage{extsizes}
\usepackage{xcolor}
\usepackage[all,cmtip]{xy}
\usepackage{multicol}
\usepackage{indentfirst}
\usepackage{latexsym}
\usepackage{bm}
\usepackage{graphicx}
\usepackage{subfigure,esint}
\usepackage{float}
\usepackage{verbatim}
\usepackage{comment}
\usepackage[top=1in, bottom=1in, left=1.25in, right=1.25in]{geometry}
\usepackage{epsfig,dsfont,amsthm,amsfonts,amsbsy}
\usepackage[colorlinks]{hyperref}
\usepackage[toc,page]{appendix}
\usepackage{tikz-cd}
\usepackage{amsthm,thmtools,xcolor}

\usepackage{adjustbox}
\usepackage{enumitem}
\setlist{nosep}
\setlist[itemize]{leftmargin=*}
\setlist[enumerate]{leftmargin=*,align=left,nolistsep}

\newtheorem{theorem}{Theorem}[section]

\newtheorem{claim}[theorem]{Claim}

\newtheorem{corollary}[theorem]{Corollary}

\newtheorem{definition}[theorem]{Definition}
\newtheorem{question}[theorem]{Question}

\newtheorem{lemma}[theorem]{Lemma}

\newtheorem{proposition}[theorem]{Proposition}

\theoremstyle{definition}
\newtheorem{example}[theorem]{Example}
\newtheorem{remark}[theorem]{Remark}

\input{convention.tex}

\usepackage{calligra}
\DeclareMathAlphabet{\mathcalligra}{T1}{calligra}{m}{n}

\declaretheoremstyle[
headfont=\color{black}\normalfont\bfseries,
bodyfont=\color{black}\normalfont\itshape,
]{colored}

\usepackage[utf8]{inputenc}
\usepackage[english]{babel}
\usepackage{fancyhdr}
\usepackage{accents}

\begin{document}
	\title[The Spectral base and quotients of bounded symmetric domains]{The Spectral base and quotients of bounded symmetric domains by cocompact lattices}

   \author{Siqi He} 
	\address{Siqi He, Institute of Mathematics, Academy of Mathematics and Systems Science, Chinese Academy of Sciences, Beijing, 100190, China}
	\email{\href{sqhe@amss.ac.cn}{sqhe@amss.ac.cn}}	
     \author{Jie Liu} %
\address{Jie Liu, Institute of Mathematics, Academy of Mathematics and Systems Science, Chinese Academy of Sciences, Beijing, 100190, China}
\email{\href{jliu@amss.ac.cn}{jliu@amss.ac.cn}}

	\author{Ngaiming Mok} 
	\address{Ngaiming Mok, Department of Mathematics, The University of Hong Kong, Pokfulam Road, Hong Kong}
	\email{\href{nmok@hku.hk}{nmok@hku.hk}}	
	\maketitle
	\begin{abstract}
In this article, we explore Higgs bundles on a projective manifold $X$, focusing on their spectral bases —-- a concept introduced by T.~Chen and B.~Ng\^{o}. The spectral base is a specific closed subscheme within the space of symmetric differentials. We observe that if the spectral base vanishes, then any reductive representation $\rho: \pi_1(X) \to \text{GL}_r(\mathbb{C})$ is both rigid and integral. Additionally, we prove that for $X=\Omega/\Gamma$, a quotient of a bounded symmetric domain $\Omega$ of rank at least $2$ by a torsion-free cocompact irreducible lattice $\Gamma$, the spectral base indeed vanishes, which generalizes a result of B.~Klingler.
	\end{abstract}

\section{Introduction}
\label{sec_introduction}
Let $X$ be a projective manifold and denote by $\Omega_X^1$ its cotangent bundle. A symmetric differential is a holomorphic section of $\Sym^i\Omega_X^1$ for some positive integer $i$. According to Hodge theory, the cotangent bundle $\Omega_X^1$ admits non-zero holomorphic sections if and only if the abelianization of $\pi_1(X)$ is infinite. However, the relationship between symmetric differentials and fundamental groups remains mysterious, and this question has been introduced by F.~Severi \cite{severi1959géométrie} and H.~Esnault in various contexts.

\begin{question}[\protect{Severi, Esnault}]
\label{question_main}
    What is the relationship between the symmetric differentials and the fundamental group? 
\end{question}

The question mentioned above has been approached from different perspectives in \cite{Klingler2013} and \cite{brunebarbeklingler2013symmetric}, where the main tool is the non-abelian Hodge correspondence, which identifies the moduli space of Higgs bundles and character varieties. In this article, we
attempt to gain insight into this question by adopting a novel approach proposed
by T.~Chen and B.~Ngô in \cite{ChenNgo2020}.

The moduli space of Higgs bundles has been extensively utilized to investigate the topology and geometry of character varieties \cite{Simpson1991, Simpson1992, simpson1994moduli, simpson1996hodge}. For a fixed rank $r$, consider the moduli stack $\MMH^{\textup{stack} {\color{black}, r}}$ of Higgs bundles {\color{black} of} rank $r$. The Hitchin morphism, introduced by Hitchin \cite{hitchin1987self, hitchin1992lie}, plays a crucial role in understanding the moduli space. By taking invariant polynomials, the Hitchin morphism is a map
    \[
    \sh_X: \MMH^{\textup{stack}{\color{black}, r}} \rightarrow \mathcal{A}_X^r := \bigoplus_{i=1}^r H^0(X,\Sym^i \Omega_X^1).
    \]
    Moreover, the affine space $\MA_X^r$ is called the \emph{Hitchin base}. 

When $\dim\;X\geq 2$, a Higgs bundle $(\msE,\vp)$ must satisfy an extra integrability condition $\varphi \wedge \varphi=0$, which makes the Hitchin morphism not surjective in general. T.~Chen and B.~Ngô introduced in \cite{ChenNgo2020} the \emph{spectral base} $\MS_X^r$ as a closed subscheme of the Hitchin base $\mathcal{A}_X^r$. They demonstrated that the integrability condition leads to the Hitchin morphism $\sh_X$ factoring through the natural inclusion $\iota_X: \MS_X^r \rightarrow \mathcal{A}_X^r$.

We denote by $\mfR^{\GL_r(\mbC)}$ the $\GL_r(\mbC)$ character variety.  Recall that a representation $[\rho]\in \mfR^{\GL_r(\mbC)}$ is called \emph{rigid} if it is an isolated point, and $\mfR^{\GL_r(\mbC)}$ is termed \emph{rigid} if every representation $[\rho]\in \mfR^{\GL_r(\mbC)}$ is isolated. In particular, since the character variety is a complex affine variety, the character variety is rigid if and only if it is zero-dimensional. Rigid representations are of particular interest. It has been shown in \cite{Simpson1992} that rigid representations are $\mbC$-variations of Hodge structures ($\mbC$-VHS for short). Furthermore, it is conjectured that rigid representations (local systems) originate from geometric sources \cite{Simpson1991}.

A representation $\rho:\pi_1(X)\rightarrow \GL_r(\mbC)$ is called \emph{integral} if it is conjugate to a representation $\pi_1(X)\rightarrow \GL_r(\MO_K)$, where $K$ is a number field and $\MO_K$ is the ring of integers of $K$. The motivation to consider rigid representations comes from \emph{Simpson's integrality conjecture}, which predicts that any rigid representation is integral. This conjecture has been confirmed by H.~Esnault and M.~Groechenig for cohomological rigid local systems (see also Remark \ref{r.EsnaultTalk}).

Using the spectral base, we obtain the following result, which generalizes \cite{Simpson1992, arapura2002higgs, Klingler2013}, see also \cite{zuo1996kodaira, katzarkov1997shafarevich,  Eyssidieux2004, cadorel2022hyperbolicity} for various generalizations.

\begin{theorem}\label{thm_main_theorem_general_variety}
    Let $X$ be a projective manifold such that $\MS^r_X=0$ for some $r\geq 1$. Then the following statements hold:
    \begin{enumerate}
        \item Any reductive representation $\rho:\pi_1(X)\to \GL_r(\mbC)$ is rigid and integral. Moreover, it is a complex direct factor of a $\mbZ$-variation of Hodge structures.
        \item Let $F$ be a non-Archimedean local field. Then any reductive representation $\rho:\pi_1(X)\to \GL_r(F)$ has bounded image.
    \end{enumerate}
\end{theorem}

Therefore, we will be particularly interested in the varieties with vanishing spectral bases. In view of Margulis' superrigidity \cite{Margulis1991}, examples of varieties with rigid character varieties are Hermitian locally symmetric spaces of higher rank. Thus we would like to understand Simpson's integrality conjecture from the perspective of Higgs bundles and spectral varieties in the case of Hermitian locally symmetric spaces with rank $\geq 2$, following the approach of Klingler \cite{Klingler2013}. In particular, we obtain the following:

\begin{theorem}\label{t.SpBaseArithVar}
Let $\Omega=\Omega_1\times \cdots \times \Omega_m$ be a bounded symmetric domain of rank $\geq 2$ together with its decomposition into irreducible factors. Let $\Gamma\subset {\rm Aut}(\Omega)$ be a torsion-free irreducible cocompact lattice, and write $X\coloneqq \Omega/\Gamma$. Then $\MS^r_X=0$ for any $r\geq 1$.
\end{theorem}

By \cite[Appendix IV, Proposition 3]{Mok1989}, the cotangent bundle of a compact quotient $X$ of an irreducible bounded symmetric domain $\Omega$ of rank $\geq 2$ by a torsion-free lattice $\Gamma \subset {\rm Aut}(\Omega)$ is big (aka \emph{almost ample} in \cite{Mok1989}), and a similar proof yields the same when $\Omega$ is reducible and of rank $\ge 2$, and the lattice $\Gamma$ is irreducible  --- see also \cite[Theorem 1.1]{brunebarbeklingler2013symmetric}. In particular, for $k$ sufficiently large and sufficiently divisible, we have
        \[
        h^0(X,\Sym^k\Omega_X^1) \sim O(k^{2\dim(X)-1}).
        \]
        So Theorem \ref{t.SpBaseArithVar} is somewhat surprising because the dimension of the Hitchin space $\MA^r_X$ is very large for $r\gg 1$, while the closed subset $\MS_X^r$ of $\MA_X^r$ is just a single point. This may also be seen as a strong piece of evidence that Theorem \ref{thm_main_theorem_general_variety} may be applicable in other interesting cases.

On the other hand, we remark that it has been shown by B.~Klingler in \cite[Theorem 1.6]{Klingler2013} the vanishing of the Hitchin base in some small ranges for compact quotients of certain irreducible bounded symmetric domains. Its proof is based on classical plethysm, a vanishing theorem of the last author \cite[p.~205 and p.~211]{Mok1989} and a case-by-case argument depending on the types of $\Omega$. However, one cannot expect to generalize its proof to large $r$: the dimension of $\MA_X^r$ can be as large as possible for $r\gg 1$ as explained above. Our argument is purely geometric and provides a unified approach to all $\Omega$ and all $r$. It relies on a Finsler metric rigidity theorem of the last author proved in \cite{mok2004extremal}.

Finally we note that Theorem \ref{t.SpBaseArithVar} cannot be strengthened to the rank-1 case, i.e., the case where $\Omega=\mbfB_{\mbC}^n$: there exists ball quotients satisfying $H^0(X,\Omega_X^1)\not=0$ and then one can easily construct a non-zero element in $\MS_X^r$ for any $r\geq 1$ in this case. 

Thanks to Theorem \ref{thm_main_theorem_general_variety}, we conclude from the vanishing of the spectral base the following result. 
\begin{corollary}
\label{c.RigInte}
In the notation of Theorem \ref{t.SpBaseArithVar} and under the assumption given there, any reductive representation $\rho:\Gamma\to \GL_r(\mbC)$ is rigid and integral for any $r\geq 1$. Moreover, it is a complex direct factor of a $\mbZ$-variation of Hodge  structures.
\end{corollary}

The rigidity result allows us to recover Margulis' rigidity result \cite{Margulis1991} in the case of cocompact lattices. It can be seen as a strengthening of Klingler's result \cite{Klingler2013} for any $r$ and any $\Omega$ of rank $\geq 2$. Moreover, we note that Margulis' rigidity only concerns representations, which are identified with topologically trivial Higgs bundles via the non-abelian Hodge correspondence, while we can also obtain a result on general Higgs bundles from the vanishing of the spectral base as in the corollary given below. This can be used to help us understand the analytic aspect of the Hermitian-Yang-Mills equation, for which we refer the reader to \cite{he2020behavior} for a discussion.

\begin{corollary}
\label{cor_nilpotent_Higgs}
In the notation of Theorem \ref{t.SpBaseArithVar}
and under the assumption given there, every Higgs bundle over $X$ is nilpotent. 
\end{corollary}

\textbf{Acknowledgements.} 
	The authors also wish to express their gratitude to a great many people for their interest and helpful comments. 
	Among them are Shan-Tai Chan, Ya Deng, Ziyang Gao, Andriy Haydys, Thomas Walpuski, Pengyu Yang, Kang Zuo. S.~He is supported by National Key R\&D Program of China (No.2023YFA1010500) and NSFC grant (No.12288201). J.~Liu is supported by the R\&D Program of China (No.2021YFA1002300), the NSFC grant (No.12288201), the CAS Project for Young Scientists in Basic Research (No.YSBR-033) and the Youth Innovation Promotion Association CAS. N.~Mok is supported by the GRF grant 17306523 of the Hong Kong Research Grants Council.

\section{Higgs bundles and Hitchin morphism}

In this section, we delve into the non-abelian Hodge correspondence and explore the Hitchin morphism over a projective variety. This subject has garnered extensive attention in various notable works, including \cite{hitchin1987self, Simpson1988Construction, simpson1994moduli, simpson1994moduli2}. Readers interested
in the topic may consult the surveys \cite{garcia2015introduction, Wentworth2016,schaposnik2018introduction}.

\subsection{Higgs bundle and non-abelian Hodge correspondence}

Let $X$ be a projective manifold, and denote by $\Omega_X^1$ the holomorphic cotangent bundle of $X$. We collect some basic definitions and facts about Higgs sheaves/bundles, which we refer \cite{Simpson1992,BandoSiu1994,BiswasSchumacher2009,LiZhangZhang2017}.
	
	\begin{definition}
		A Higgs sheaf on $X$ is a pair $(\msE,\vp)$, where $\msE$ is a torsion-free coherent sheaf on $X$ and $\vp\colon \MSE \rightarrow \MSE\otimes \Omega_X^1$, called a Higgs field, such that the composed morphism
        \[
        \MSE \xrightarrow{\varphi} \MSE \otimes \Omega_X^1 \xrightarrow{\varphi\otimes \id} \MSE\otimes \Omega_X^1\otimes \Omega_X^1 \xrightarrow{\id\otimes \wedge} \MSE\otimes \Omega_X^2
        \]
        vanishes. Following tradition, the composed morphism will be denoted by $\vp\wedge\vp$ and the equation $\vp\we \vp=0$ is called the Higgs equation. 
	\end{definition}
	
    Given a Higgs sheaf $(\MSE,\varphi)$ over a projective manifold $X$, a coherent subsheaf $\MSF\subset \msE$ is said to be \emph{$\vp$-invariant} if and only if $\vp(\MSF)\subset \MSF\otimes\Omega_X^1$. Now we can introduce the concept of slope stability. Recall that, given a torsion-free coherent sheaf $\MSE$ on an $n$-dimensional projective manifold $X$, the \emph{slope $\mu(\MSE)$} of $\MSE$ with respect to $\omega$ is defined to be 
    $$
    \mu(\MSE)=\frac{\deg_{\omega}(\MSE)}{\rank\; \MSE}=\frac{c_1(\MSE)\cdot \omega^{n-1}}{\rank\;\MSE}.
    $$ 

	\begin{definition}
		A Higgs sheaf $(\msE,\vp)$ is called stable (resp. semistable) if and only if for {\color{black}any} $\vp$-invariant coherent subsheaf $\MSF\subset \msE$, with $0<\rank(\MSF)<\rank(\MSE)$, we have 
        \[
        \mu(\MSF)<\mu(\msE)\quad (\textup{resp.}\,\mu(\MSF)\leq \mu(\MSE)).
        \]
        A Higgs sheaf $(\msE,\vp)$ is called polystable if $(\msE,\vp)$ is semistable and 
       \[
       (\msE,\vp)\cong (\msE_1,\vp_1)\oplus\cdots \oplus (\msE_r,\vp_r) {\color{black},}
       \]
       where $(\msE_i,\vp_i)$ are stable Higgs sheaves with the same slope.
	\end{definition}
We will now focus on Higgs bundles, equivalently locally free Higgs sheaves, and will return to the general notion of Higgs sheaves in \S\,\ref{ss.spectralvariety}, where such sheaves are constructed from spectral varieties defined by spectral data.

Let $E$ be a complex smooth vector bundle over $X$. We write $\Omega^{p,q}(E)$ for the complex vector space of $E$-valued $(p,q)$-forms on $X$. In the sequel of this paper, we will naturally identify the holomorphic structures on $E$ with the $\bar{\pa}$-operators $\bar{\pa}_E\colon \Omega^{p,q}(E)\to \Omega^{p,q+1}(E)$ satisfying the integrability condition $\bar{\pa}_E^2=0$. We denote by $\msE:=(E,\bar{\pa}_E)$ the holomorphic vector bundle with the holomorphic structure defined by $\bar{\pa}_E$ if there is no confusion.
 
Let $g\in \Aut(E)$. Then $g$ acts on the Higgs bundles $\MSE=((E,\bar{\pa}_E),\vp)$ by $g\cdot(\bar{\pa}_E,\vp)=(g^{-1}\circ\bar{\pa}_E\circ g,g^{-1}\circ\Phi\circ g).$ We define the moduli stack of polystable Higgs bundles of rank $r$ as
\begin{equation}
\begin{split}
    \MM^{{\rm stack}, r}_{\Higgs}\coloneqq \{(\msE,\vp)|(\msE,\vp)\;\mathrm{polystable}\}/\Aut(E).
\end{split}
\end{equation}
	
A complex smooth vector bundle $E$ is said to be \emph{topologically trivial} if all the Chern classes of $E$ in $H^{*}(X;\mathbb{Q})$ vanish. Let $[(\msE,\vp)]$ be the equivalence class of $(\msE,\vp)$ in the orbit of $\Aut(E)$. Under S-equivalence of $\Aut(E)$ action, a semi-stable topologically trivial Higgs bundle $[(\msE,\vp)]$ is polystable. Following \cite[Proposition 6.6]{simpson1994moduli2}, we define the Dolbeault moduli space $\MMD^r$ as the moduli space parametrizing topologically trivial polystable Higgs bundles on $X$, which is a quasiprojective variety. It follows from \cite[Proposition 3.4]{Simpson1988Construction} that a polystable Higgs bundle $(\msE,\vp)$ is topologically trivial if and only if $c_1(\msE)\cdot \omega^{n-1}=0$ and $c_2(\msE)\cdot \omega^{n-2}=0$.

Let $X$ be a projective variety. The $\GL_r(\mbC)$ character variety $\mfR^{\GL_r(\mbC)}$ of $X$ is defined to be the set of conjugacy classes of reductive representations of the fundamental group given by
\begin{equation}
    \mfR^{\GL_r(\mbC)}\coloneqq \{\rho\colon \pi_1(X)\to \GL_r(\mbC)\mid \rho\;\mathrm{reductive}\}/\sim.
\end{equation}

\begin{theorem}[\protect{\cite{hitchin1987self,donaldson1987twisted,Simpson1988Construction,corlette1988flat}}]
\label{t.Homeo}
    There exists a bijective map 
    \begin{equation}
    \label{eq_NAH_map}
        \xi:\MMD^r\to \mfR^{\GL_r(\mbC)},\;
    \end{equation}
    which is real analytic over the smooth locus.
\end{theorem}
 
\subsection{Hitchin morphism}
The Hitchin morphism is a useful tool to study the moduli space of Higgs bundles. In this subsection, we will introduce the Hitchin morphism for projective manifolds, following \cite{hitchin1987self, hitchin1987stable, simpson1994moduli}. Let $X$ be a projective manifold. Then the \emph{Hitchin base} of $X$ with rank $r$ is defined to be 
	\begin{equation}
		\MA^r_X\coloneqq \bigoplus_{i=1}^rH^0(X,\Sym^i\Omega_X^1).
	\end{equation}
	The \emph{Hitchin morphism} for the moduli stack of Higgs bundle is defined as follows: 
	\begin{equation}
		\begin{split}
  \label{e.hitchinmap}
			\sh_X:\MMH^{\textup{stack},r} \to \MA^r_X,\quad [(\MSE,\vp)]\mapsto (\Tr(\vp),\Tr(\vp^2)\cdots,\Tr(\vp^r)).
		\end{split}
	\end{equation}

	\begin{theorem}[\protect{\cite{hitchin1987self,simpson1994moduli}}]
 \label{thm_Hitchinmap_proper}
		The restriction $\sh_X|_{\MMD^ {r}}:\MMD^r\to \MA^r_X$ is proper, and it is also surjective in the case where $\dim(X)=1$. 
	\end{theorem}

\subsection{Spectral base}
We briefly recall the definition of the spectral base, which was introduced by T.~Chen and B.~Ng\^o in \cite{ChenNgo2020}.
\begin{definition}
\label{d.spectralbase}
    The spectral base $\MS^r_X$ is the subset of $\MA^r_X$ consisting of the elements $\textbf{s}=(s_1,\dots,s_r)\in \MA^r_X$ such that for any point $x\in X$, there exist $r$ elements $\omega_1,\dots, \omega_r\in \Omega_{X,x}^1$ satisfying $s_i(x)=\sigma_i(\omega_1,\dots,\omega_r)$, where $\sigma_i$ is the $i$-th elementary symmetric polynomial in $r$ variables. Moreover, an element $\textbf{s}\in \MS^r_X$ is called a spectral datum.
\end{definition}

    Let $V$ be a complex vector space of dimension $n$ and let $\Chow^r(V)$ be the Chow variety of zero cycles of length $r$ on $V$. By \cite[Theorem 4.1]{ChenNgo2020}, the following natural map
    \[
    \Chow^r(V)  \rightarrow V\times \Sym^2 V \times \dots \times \Sym^r V,\quad [v_1,\dots, v_r] \mapsto (\sigma_1, \sigma_2,\dots, \sigma_r),
    \]
    is a closed embedding and thus it induces the following closed embedding
    \begin{equation}
    \label{e.closedembeddingChow^r}
        \Chow^r(\Tot(\Omega_X^1)/X) \hookrightarrow \Tot(\Omega_X^1)\times_X \Tot(\Sym^2\Omega_X^1)\times_X \dots \times_X \Tot(\Sym^r \Omega_X^1),
    \end{equation}
    where $\Tot(\bullet)$ denotes the total space of the corresponding vector bundle and the space $\Chow^r(\Tot(\Omega_X^1)/X)$ is the relative Chow space of zero cycles of length $r$. In particular, under this closed embedding, the spectral base $\MS^r_X$ can be identified with the space of sections $\sigma:X\rightarrow \Chow^r(\Tot(\Omega_X^1)/X)$ and so $\MS^r_X$ is a closed subset of $\MA^r_X$.

    The following observation shows that it suffices to check the condition in Definition \ref{d.spectralbase} over general points to see whether an element $\textbf{s}\in \MA^r_X$ is a spectral datum.

\begin{lemma}
\label{l.birationalinvarance}
    Let $\textbf{s}\in \MA^r_X$ be an element. If there exists a dense Zariski open subset $X^{\circ}$ of $X$ such that $\textbf{s}$ satisfies the condition in Definition \ref{d.spectralbase} for any point $x\in X^{\circ}$, then $\textbf{s}\in \MS^r_X$.
\end{lemma}

\begin{proof}
    Denote by 
    \[
    \sigma:X\rightarrow \Tot(\Omega_X^1)\times_X \Tot(\Sym^2 \Omega_X^1)\times_X \dots\times_X \Tot(\Sym^r \Omega_X^1)
    \]
    the section corresponding $\textbf{s}$. Then the image of $\sigma$ is contained in $\Chow^r(\Tot(\Omega_X^1)/X)$ over $X^{\circ}$ by our assumption and \eqref{e.closedembeddingChow^r}. However, since \eqref{e.closedembeddingChow^r} is a closed embedding and the image of $\sigma$ is irreducible, the image of $\sigma$ is contained in $\Chow^r(\Tot(\Omega_X^1)/X)$. Hence, $\sigma$ is a section of $\Chow^r (\Tot(\Omega_X^1)/X) \rightarrow X$.
\end{proof}

\begin{remark}
    \begin{enumerate}
        \item As an immediate consequence of Lemma \ref{l.birationalinvarance}, one can derive the birational invariance of the spectral base $\MS^r_X$ as proved by L.~Song and H.~Sun in \cite[Theorem 5.3]{SongSun2024}.

        \item For any positive integers $r'<r$, there exists a natural inclusion $\MS_X^{r'}\subset \MS_X^{r}$ defined as following:
        \[
        \textbf{s}'=(s_1,\dots,s_{r'}) \mapsto \textbf{s}:=(s'_1,\cdots,s'_{r},0,\cdots,0) \in \MA_X^r.
        \]
        We only need to show that $\textbf{s}$ is contained in $\MS_X^r$. Indeed, given an arbitrary point $x\in X$, let $w_1,\dots, w_{r'}\in \Omega_{X,x}^1$ be the points such that $s'_i=\sigma_i(w_1,\dots,w_{r'})$. Then one can easily conclude by considering the set $\{w_1,\dots,w_{r'},0^{r-r'}\}$. In particular, if $\MS_X^r=0$, then so is $\MS_X^{r'}$ for any $r'\leq r$.
    \end{enumerate}
\end{remark}

\begin{proposition}[\protect{\cite[Proposition 5.1]{ChenNgo2020}}]
		\label{p.CN-rankone}
		The Hitchin morphism $\sh_X:\MMH^{\textup{stack},r}\to \MA^r_X$ factors through the natural inclusion map $\iota_X:\MS^r_X\rightarrow \MA^r_X$. In other words, there exists a map $\ssd_X:\MMH^{\textup{stack},r}\to \MS^r_X$ such that the following diagram commutes: 
		\begin{equation}
           \label{e.spectralmorphism}
			\begin{tikzcd}[column sep=large,row sep=large]
				\MMH^{\textup{stack},r}\arrow[d,"{\ssd_X}" left] \arrow{dr}{\sh_X} &  \\
				\MS^r_X \arrow[r,"{\iota_X}" below] & \MA^r_X .
			\end{tikzcd}
		\end{equation}
		The map $\ssd_X$ is called the spectral morphism
	\end{proposition}

 \begin{proof}
     Let $(\MSE, \vp)$ be a rank $r$ Higgs bundle with $\Tr(\vp^i) = s_i \in H^0(X, \Sym^i\Omega_X^1)$. Given an arbitrary point $x\in X$, let $dz^1, dz^2, \cdots, dz^n$ be a frame of $\Omega^{1}_X$ at $x \in X$. If we write $\vp(x) = \sum_{i=1}^n A_i dz^i$, then the condition $\vp \wedge \vp = 0$ implies that $[A_i, A_j] = 0$ for any $1\leq i\leq j$. Thus $A_i$'s can be simultaneously upper-triangularized and so is $\varphi(x)$ as an $r\times r$-matrix with values in one forms. In particular, after changing local coordinates, we may assume that $\varphi(x)$ is a upper triangular matrix and let $\omega_1,\dots \omega_r\in \Omega^1_{X,x}$ be its diagonal elements. By the definition of the Hitchin morphism, we have $s_i(x)=\sigma_i(\omega_1,\dots,\omega_r)$ and hence we are done.
 \end{proof}

 In \cite{ChenNgo2020}, T.~Chen and B.~Ng\^o conjectured that $\ssd_X$ is surjective. This conjecture has been confirmed in \cite{ChenNgo2020} and \cite{SongSun2024} for smooth projective surfaces, in \cite{heliu2023spectralvariety} for rank two case and studied in \cite{bolognese2023p} for abelian variety. However, {in general} the moduli stack $\MMH^{{\rm stack}, r}$ may be much larger than $\MMD^r$. We are in particular interested in the image of the restriction of the spectral morphism to the Dolbeault moduli space, which leads to the following definition.

 \begin{definition}
The Dolbeault spectral base $\MS^r_{X;\Dol}$ is defined to be the image $\ssd_X(\MMD^r)$.
\end{definition}

Clearly we have the natural inclusions $\MS^r_{X;\Dol} \subset \MS^r_X \subset \MA^r_X$ and both of them are strict in the general case (see \cite[Example 3.4]{heliu2023spectralvariety}). Moreover, since by Theorem 2.4 the restriction $\sh_X|_{\MMD^r}$ is proper, the Dolbeault spectral base $\MS^r_{X;\Dol}$ is a closed subset of $\MS^r_X$.

\subsection{Spectral variety and its decomposition}
\label{ss.spectralvariety}

Let $p:\Tot(\Omega_X^1)\rightarrow X$ be the natural projection. Given a spectral datum $\textbf{s}=(s_1,\dots,s_r)\in \MS^r_X$, the \emph{spectral variety $X_{\textbf{s}}$ corresponding to $\textbf{s}$} is the closed subscheme of $\Tot(\Omega_X^1)$ defined as follows
\[
X_{\textbf{s}}\coloneqq \{\lambda^r-s_1\lambda^{r-1}+\dots+(-1)^{r-1}s_{r-1}\lambda+(-1)^r s_r =0\},
\]
where $\lambda\in H^0(\Tot(\Omega_X^1),p^*\Omega_X^1)$ is the Liouville form and for any $1\leq i\leq r$, the term $s_i\lambda^{r-i}$ is regarded as an element of $H^0(\Tot(\Omega_X^1),\Sym^r p^*\Omega_X^1)$. In particular, the subscheme $X_{\textbf{s}}$ is locally defined by $\binom{n+r-1}{r-1}$ equations. To understand the spectral variety, we introduce the notion of multivalued holomorphic $1$-forms --- see \cite[Definition 5.8]{cadorel2022hyperbolicity}.

\begin{definition}
    Let $X$ be a projective manifold, and let $\{U_i\}$ be an open covering of $X$ in the Euclidean topology. A multivalued holomorphic $1$-form is a collection of multisets $\{\omega_{i1},\cdots,\omega_{ir}\}$ where $\omega_{il}\in H^0(U_i,\Omega_X^1|_{U_i})$ and over $U_i\cap U_j$, we have 
    $\{\omega_{i1},\cdots,\omega_{ir}\}=\{\omega_{j1},\cdots,\omega_{jr}\}$ counted with multiplicity. We write $[(\omega_1,\cdots,\omega_r)]$ to denote a multivalued holomorphic $1$-form. 
\end{definition}

Let $\widehat{X}_{\mbfs}$ be the reduced scheme underlying $X_{\mbfs}$; that is, $\widehat{X}_{\mbfs}$ is the same topological space as $X_{\mbfs}$, but with the reduced structure sheaf. Then the natural morphism $\pi:\widehat{X}_{\mbfs}\rightarrow X$ is surjective and finite. In particular, there exists a dense Zariski open subset $X^{\circ}$ of $X$ such that $\widehat{X}^{\circ}_{\mbfs}\coloneqq \pi^{-1}(X^{\circ})\rightarrow X_{\circ}$ is an unramified finite covering. Moreover, we can also write
\[
\widehat{X}^{\circ}_{\mbfs} = \bigcup_{k=1}^m \widehat{X}^{\circ}_{\mbfs,k}
\]
for the decomposition of $\widehat{X}^{\circ}_{\mbfs}$ into irreducible components. Since $\widehat{X}^{\circ}_{\mbfs}\rightarrow X^{\circ}$ is unramified, the decomposition above is actually a disjoint union. On the other hand, one can easily see that each irreducible component $\widehat{X}^{\circ}_{\mbfs,k}$ defines a multivalued holomorphic $1$-form $[(\omega_{1}^k,\dots,\omega^k_{r_k})]$ over $X^{\circ}$ whose local representatives have no multiple elements. 

Now we can define the multiplicity of $\widehat{X}^{\circ}_{\mbfs,k}$ in the spectral variety $X_{\mbfs}$. For any $k$ and $l$, we define the multiplicity $m(k,l)$ of the section $\omega^k_{l}$ to be its multiplicity as a root of the equation
\[
\lambda^r - s_1\lambda^{n-1} + \dots + (-1)^{n-1} s_{n-1} \lambda + (-1)^n s_n =0.
\]

\begin{lemma}
\it The multiplicity $m(k,l)$ is independent of $l$.
\end{lemma}
\begin{proof}
    Define a function $\mu: \widehat{X}_{\bf s}^o \to \mathbb N$ as follows. Any point $w \in  \widehat{X}_s^o$ is a tangent covector at $x = \pi(w)$ of the form $w = \omega_j(x)$ for some $j$, $1 \le j \le r$.  Define now $\mu(w)$ to be the multiplicity of $\omega_j$ at $x \in X^o$.  From the definition of the multiplicity it follows that $\mu$ is locally constant on $\widehat{X}_{\bf s}^o$, hence it must be constant on each connected component 
$\widehat{X}_{{\bf s},k}^o$, $1 \le k \le m$.
\end{proof}

We shall denote $m(k,l)$ by $m(k)$. Then we have $\sum_{k=1}^m m(k)r_k=r$. Let $\widehat{X}_{\mbfs,k}$ be the closure of $\widehat{X}^{\circ}_{\mbfs,k}$ in $\widehat{X}_{\mbfs}$ and let $\hat{\pi}_k:\widehat{X}_{\mbfs,k}\rightarrow X$ be the natural morphism. We define
\[
\hat{\MSF}_k\coloneqq \hat{\pi}_{k*}\MSO_{\hat{X}_{\mbfs,k}}.
\]
Since $\hat{X}_{\mbfs,k}$ is integral, the structure sheaf $\MSO_{\hat{X}_{\mbfs,k}}$ is torsion-free. So $\hat{\MSF}_k$ is also torsion-free and it carries a natural Higgs field $\psi_k$ defined as follows:
\begin{equation}
\label{e.inducedHiggs}
    \psi_{k}: \hat{\MSF}_k=\hat{\pi}_{k*}\MSO_{\widehat{X}_{\mbfs,k}} \xrightarrow{\times \lambda} \hat{\pi}_{k*}\left(\MSO_{\widehat{X}_{\mbfs,k}}\otimes \hat{\pi}_k^*\Omega_X^1\right) = \hat{\MSF}_k\otimes \Omega_X^1.
\end{equation}

Recall that a coherent sheaf $\MSF$ over a complex manifold is said to be a \emph{normal sheaf} if Hartogs' extension across subvarieties of codimension $\geq 2$ holds, and $\MSF$ is said to be a \emph{reflexive sheaf} if $\MSF^{**}=\MSF$. By \cite[Chapter 2, Lemma 1.1.12]{OkonekSchneiderSpindler2011}, a coherent sheaf on a complex manifold is reflexive if and only if it is normal and torsion-free. 

\begin{proposition}
    Let $\MSE_k$ be the reflexive hull of $\hat{\MSF}_k$, i.e., $\MSE_k=\hat{\MSF}_k^{**}$.
    \begin{enumerate}
        \item The Higgs field $\psi_k$ extends to a Higgs field $\varphi_k$ on $\MSE_k$.

        \item The sheaf $\MSE_k$ carries a natural $\MSO_X$-algebra structure.
    \end{enumerate}
\end{proposition}

\begin{proof}
    There exists a dense Zariski open subset $\iota:U \hookrightarrow X$ such that $X\setminus U$ is of codimension $\ge 2$ in $X$ and such that $\hat{\MSF}_k$ is locally free on $U$, so that $\hat{\MSF}_k|_{U}\cong \MSE_k|_U$. Hence, $\mathscr{E}_k|_U$ inherits from $\hat{\MSF}_k$ the structure of an $\mathscr O_U$-algebra. On the other hand, since $\MSE_k$ is reflexive, we have 
    \[
    \MSE_k\cong \iota_*(\MSE_k|_U)=\iota_*(\hat{\MSF}_k|_U).
    \]
    In particular, the restriction
    \[
    \MSE_k|_U\cong \hat{\MSF}_k|_U \xrightarrow{\psi_k|_U} \hat{\MSF}_k|_U\otimes \Omega_U^1 \cong \MSE_k|_U\otimes \Omega_U^1
    \]
    extends to a morphism $\varphi_k:\MSE_k\rightarrow \MSE_k\otimes \Omega_X^1$ and  we have an extension of the structure of the $\MSO_U$-algebra structure on $\MSE_k|_U$ to an $\mathscr O_X$-algebra structure on $\MSE_k$. 
\end{proof}

\begin{remark}
    \label{r.CohenMacaulay}
    Let $\widetilde{X}_{k}$ be the variety defined as $\textup{Spec}_{\MSO_X}\MSE_k$. Then we have a natural finite birational morphism $\widetilde{X}_{\mbfs,k}\rightarrow \widehat{X}_{\mbfs,k}$, which is an isomorphism in codimension one. Moreover, since $\MSE_k$ is locally free in codimension two, the variety $\widetilde{X}_{\mbfs,k}$ is Cohen-Macaulay in codimension two (\cite[IV, D, Corollaire 2]{Serre1965}).  In particular, the variety $\widetilde{X}_{\mbfs,k}$ is Cohen-Macaulay if $\dim(X)=2$.
\end{remark}

Since $X$ is smooth and $\MSE_k$ is reflexive, there exists a dense Zariski open subset $X^{\circ\circ}$ of $X$ such that $X\setminus X^{\circ\circ}$ is of codimension $\geq 3$ in $X$ and such that the restriction $\MSE_k|_{X^{\circ\circ}}$ is locally free (\cite[Chapter 2, Lemma 1.1.10]{OkonekSchneiderSpindler2011}). Denote by $\varphi_k^{\circ\circ}$ the restriction $\varphi_k|_{X^{\circ\circ}}$ and let 
\[
(s_{1k}^{\circ\circ},\cdots,s_{r_k k}^{\circ\circ})\coloneqq \sh_{X^{\circ\circ}}\big([\MSE_k|_{X^{\circ\circ}},\varphi_k^{\circ\circ}]\big) \in \bigoplus_{i=1}^{r_k} H^0(X^{\circ\circ},\Sym^i\Omega_{X^{\circ\circ}}^1).
\]
Since $\codim(X\setminus X^{\circ\circ})\geq 3$ and $\Sym^i\Omega_X^1$ is locally free, the sections $s_{ik}^{\circ\circ}$ extend to sections $s_{ik}\in H^0(X,\Sym^i\Omega_X^1)$. Let
\[
P_k(T)\coloneqq T^{r_k} - s_{1k} T^{r_k-1} + \dots + (-1)^{r_k} s_{r_k k}
\]
{\color{black}be} the corresponding characteristic polynomial. Over the dense Zariski open subset $X^{\circ}$, one can easily derive the following equality of polynomials
\begin{equation}
\label{e.decomp-char-poly}
   P(T)\coloneqq T^r - s_1 T^{r-1} + \cdots + (-1)^k s_k = \prod_{k=1}^m P_k(T)^{m(k)}. 
\end{equation}
Then it follows that the equality above actually holds over the whole $X$ by comparing the coefficients. Now let $X_{\mbfs,k}\subset \Tot(\Omega_X^1)$ be the subscheme defined by $P_k(\lambda)=0$. Then clearly we have $X_{\mbfs,k}\cap p^{-1}(X^{\circ})=\widehat{X}^{\circ}_{\mbfs,k}$ and we may also write the equality \eqref{e.decomp-char-poly} as
an equation on cycles in the form

\begin{equation}
\label{e.decomp-spectral-var}
    [X_{\mbfs}] = \sum_{k=1}^m m(k) [X_{\mbfs,k}],
\end{equation}
where for a pure-dimensional complex subspace $A$ of $X$, $[A]$ denotes the cycle in the Chow space Chow$(X)$ associated to $A$.
\begin{remark}
    An analogue of the decomposition \eqref{e.decomp-spectral-var} was obtained by L.~Song and H.~Sun in \cite[\S\,3]{SongSun2024} in the case where $X$ is a smooth projective surface and in higher dimension one can use \eqref{e.decomp-spectral-var} to reduce Chen--Ng\^o's conjecture to the case where the spectral cover $X_{\textbf{s}}$ is irreducible and generically reduced, i.e., $m=1$ and $m(1)=1$.
\end{remark}

The Hitchin morphism \eqref{e.hitchinmap} and the spectral morphism \eqref{e.spectralmorphism} can be directly extended to Higgs sheaves and one can immediately derive the following result from our argument above.

\begin{proposition}
\label{p.SurjReflexive}
    Let $X$ be a projective manifold. Given a spectral datum $\mbfs\in \MS_X^r$, there exists a reflexive Higgs sheaf $(\MSE,\varphi)$ of rank $r$ over $X$ such that $\ssd_X([(\MSE,\varphi)])=\mbfs$.
\end{proposition}

\begin{proof}
    We conclude by letting $(\MSE,\varphi)=\oplus_{k=1}^m (\MSE_k,\varphi_k)^{\oplus m(k)}$.
\end{proof}

\begin{remark}
    
\begin{enumerate}
    \item Since reflexive sheaves are locally free in codimension two (\cite[Chapter 2, Lemma 1.1.10]{OkonekSchneiderSpindler2011}), one can use Proposition \ref{p.SurjReflexive} to recover the surjectivity of the spectral morphism $\ssd_X$ proved in \cite{ChenNgo2020} and \cite{SongSun2024} in the surface case.

    \item Assume that the spectral variety $X_{\mbfs}$ is irreducible and generically reduced. As $\pi\colon X_{\mbfs}\rightarrow X$ is finite, the natural surjection $\MSO_{X_{\mbfs}}\rightarrow \MSO_{\hat{X}_{\mbfs}}$ induces a surjection
    \[
    \MSF\coloneqq \pi_*\MSO_{X_{\mbfs}} \rightarrow \hat{\pi}_* \MSO_{\hat{X}_{\mbfs}} \eqqcolon \hat{\MSF},
    \]
    where $\pi\colon X_{\mbfs}\rightarrow X$ and $\hat{\pi}\colon \hat{X}_{\mbfs}\rightarrow X$ are the natural finite morphisms, respectively. Moreover, since $\MSF\rightarrow \hat{\MSF}$ is an isomorphism over the generic point of $X$ and $\hat{\MSF}$ is torsion-free, we must have
    \[
    \hat{\MSF}\cong \MSF/\MST(\MSF),
    \]
    where $\MST(\MSF)$ is the torsion subsheaf of $\MSF$. In particular, we have $\hat{X}_{\mbfs}=\textup{Spec}_{\MSO_X} \hat{\MSF}/\MST(\MSF)$. This construction has already appeared in \cite[Remark 7.1]{ChenNgo2020} and in some special case the sheaf $\hat{\MSF}$ is already locally free, and hence reflexive --- see \cite[Example 8.1]{ChenNgo2020}. In particular, our construction of $\widetilde{X}_{\mbfs}$ can be viewed as a generalisation of that given in \cite[Remark 7.1]{ChenNgo2020}.
		
    \item Assume that $\dim(X)=2$, and $X_{\mbfs}$ is irreducible and generically reduced. Then the reflexive hull $\MSE\coloneqq \hat{\MSF}^{**}$ is locally free. So the variety $\widetilde{X}_{\mbfs}\coloneqq \textup{Spec}_{\MSO_X}\MSE$ is a finite Cohen-Macaulayfication of $X_{\mbfs}$ (Remark \ref{r.CohenMacaulay} and \cite[IV, D, Corollaire 2]{Serre1965}). On the other hand, T.~Chen and B.~Ng\^{o} has also constructed a Cohen-Macaulayfication $X^{\textup{CM}}$ of $X_{\mbfs}$ in \cite[Proposition 7.2]{ChenNgo2020} via the Hilbert scheme. Now we claim that $\widetilde{X}_{\mbfs}$ is actually isomorphic to $X^{\textup{CM}}$. Indeed, let $U$ be the largest open subset of $X$ such that $\hat{\MSF}|_U$ is locally free. Then $X\setminus U$ has codimension $\geq 2$ in $X$ and we have 
    \[
    \MSF|_U \cong \MSE|_U \cong \MSE'|_U
    \]
    as $\MSO_U$-algebras (\cite[Proposition 7.2 and Remark 7.1]{ChenNgo2020}), where $\MSE'\coloneqq \pi'_*\MSO_{X^{\textup{CM}}}$ and $\pi'\colon X^{\textup{CM}}\rightarrow X$ is the natural finite morphism. Since both $\MSE$ and $\MSE'$ are locally free, we get an isomorphism of $\MSE$ and $\MSE'$ as $\MSO_X$-algebras. Hence, there exists an isomorphism between $\widetilde{X}_{\mbfs}$ and $X^{\textup{CM}}$ satisfying the following commutative diagram:
    \[
    \begin{tikzcd}
        \widetilde{X}_{\mbfs} \arrow[dr] \arrow[rr,"\cong"]
            & 
                & X^{\textup{CM}} \arrow[dl] \\
            & \hat{X}_{\mbfs}
                &
    \end{tikzcd}
    \]
\end{enumerate} 
\end{remark}


\section{The spectral base for a quotient of a bounded symmetric domain with $\rank \geq 2$.}
Let $X$ be a quotient of a bounded symmetric domain by an irreducible torsion free cocompact lattice. In this section, we will explain the relationship between the spectral base and {\color{black}Finsler metrics}. Moreover, we will use the Finsler metric rigidity theorem of the last author \cite{mok2004extremal} to prove the vanishing of the spectral base whenever ${\rm rank}(X) \ge 2$.

\subsection{Finsler (pseudo-)metric}
Let $\MSL$ be a holomorphic line bundle over a complex manifold $X$. We briefly recall the definition of a (singular) Hermitian metric on $\MSL$.

\begin{definition}
    A singular (Hermitian) metric $h$ on a line bundle $F$ is a metric which is given in any trivialization $\theta:L|_{U}\xrightarrow{\cong} U\times \mbC$ by
    \[
    \|\xi\|_{h} = |\theta(\xi)| e^{-\varphi(x)}, x\in U, \xi\in \MSL_x,
    \]
    where $\varphi\in L^1_{\textup{loc}}(U)$ is an arbitrary locally integrable function, called the weight of the metric with respect to the trivialization $\theta$.
\end{definition}

    The \emph{curvature current} of $\MSL$ is given formally by the closed $(1,1)$-current $\frac{\sqrt{-1}}{2\pi}\Theta_{\MSL,h}=dd^c\varphi$ on $U$. The assumption $\varphi\in L^1_{\textup{loc}}$ guarantees that $\Theta_{\MSL,h}$ exists in the sense of distribution theory. Moreover, for the curvature current for is globally defined over $X$ and independent of the choice of trivialisations, and its de Rham cohomology class is the image of the first Chern class $c_1(\MSL)\in H^2(X,\mbZ)$ in $H^2_{\textup{dR}}(X,\mbR)$. If we assume in addition that $\varphi\in \mathcal{C}^{\infty}(U,\mbR)$, then $h$ is the usual smooth Hermitian metric on $\MSL$.

\begin{example}
\label{l.Current-Global-section}
    Let $D$ be an effective divisor and let $\MSL=\MSO_X(-D)$ be the ideal sheaf of $D$. Let $\MSL \rightarrow \MSO_X$ be the natural non-zero map to the trivial line bundle $\MSO_X$ over $X$. Then the standard Hermitian metric over $\MSO_X$ induces a singular Hermitian metric $h$ over $\MSL$. Indeed, let $g$ be the generator of $\MSO_X(-D)$ on an open subset $U$ of $X$, then 
    \[
    \theta(u)=\frac{u}{g}
    \]
    defines a trivialisation of $\MSO_X(-D)$ over $U$, thus our singular metric is associated to the weight $\varphi=-\log|g|$. By the Lelong--Poincar{\'e} equation, we find
    \[
    \frac{i}{2\pi} \Theta_{\MSL} = dd^{c}\varphi = -[D],
    \]
    where $[D]$ denotes the current of integration over $D$.
\end{example}

Let $\MSE$ be a holomorphic vector bundle over a complex manifold $X$. Let $\mbP(\MSE)$ be the projectivisation in the geometric sense, i.e., $\mbP(\MSE)$ parametrises the one-dimensional linear subspaces contained in the fibres of $\MSE$. Let $\MSO_{\mbP(\MSE)}(-1)\subset \pi^*\MSE$ be the dual tautological line bundle over $\mbP(\MSE)$, where $\pi:\mbP(\MSE)\rightarrow X$ is the natural projection. Given a (singular) Hermitian metric $h$ over $\MSO_{\mbP(\MSE)}(-1)$, we can define a pseudo-metric $\bar{h}$ over $\MSE$ as in the following. For any $v\in \MSE_x\setminus\{0\}$, we define
\[
\|v\|_{\bar{h}} : = \|v\|_{h},
\]
where on the right-hand side we regard $v$ as the corresponding point in the fibre of the natural projection $\MSO_{\mbP(\MSE)}(-1)\rightarrow \mbP(\MSE)$ over $[v]$. Such a metric $\bar{h}$ is called a \emph{(complex) Finsler pseudo-metric} and we call it a \emph{(complex) Finsler metric} if the metric $h$ is a smooth Hermitian metric.




\subsection{Bounded symmetric domain and Finsler metric rigidity}

  We collect some basic definitions and facts about bounded symmetric domains and we refer the interested reader to \cite{Mok1989} for more details. Let $\Omega\Subset \mbC^n$ be a bounded domain in a complex Euclidean space. We say that $\Omega$ is a \emph{bounded symmetric domain} if and only if at each $x\in \Omega$, there exists a biholomorphism $\sigma_x:\Omega\rightarrow \Omega$ such that $\sigma_x^2=\id$ and $x$ is an isolated fixed point of $\sigma_x$. In this case, the Bergman metric $ds_{\Omega}^2$ with K\"ahler form $\omega$ on $\Omega$ is K{\"a}hler--Einstein and $(\Omega,\omega)$ is a Hermitian symmetric space of {\color{black}the} non-compact type. The \emph{rank} of $\Omega$ is defined to be the rank of $(\Omega,ds_{\Omega}^2)$ as a Riemannian symmetric manifold. We say that the bounded symmetric domain $\Omega$ is \emph{irreducible} if and only if $(\Omega,ds^2_{\Omega})$ is an irreducible Riemannian symmetric manifold. Denote by $\Aut(\Omega)$ the {\color{black} group of} biholomorphic self-mappings on $\Omega$. Write $G$ for the identity component $\Aut_o(\Omega)$ and let $K\subset G$ be the isotropy subgroup at a point $o\in \Omega$, so that $\Omega=G/K$ as a homogeneous space. 

  Now we introduce the minimal characteristic bundle --- see \cite[Chapter 6,\S\,1]{Mok1989} and \cite{Mok2002a}. Let $\Omega$ be an irreducible bounded symmetric domain. Then we can identify $\Omega$ as a subdomain of its compact dual $M$ by the Borel embedding (\cite[Chapter 3, \S\,3]{Mok1989}). Let $G^{\mbC}$ be the automorphism group of $M$ and $P\subset G$ be the isotropy subgroup at $o$. Then $G^{\mbC}\supset G$ is a complexification of $G$. Consider the action of $P$ on $\mbP(T_o M)$. There are exactly $r$ orbits $\MO_k\subset \mbP(T_o M)=\mbP(T_o\Omega)$, $1\leq k\leq r$, such that the topological closures $\bar{\MO}_k$ form an ascending chain of subvarieties of $\mbP(T_o M)$ with $\bar{\MO}_r=\mbP(T_o M)$. In particular, the variety $\MO_1$ is the unique closed orbit, which is thus a homogeneous projective submanifold of $\mbP(T_o M)$. Moreover, the submanifold $\MO_1\subset \mbP(T_o M)$ is nothing {\color{black} other than} the \emph{variety of minimal rational tangents} (VMRT) of $M$ at $o$, i.e., the variety of tangent directions {\color{blue} at $o$} of projective lines on $M$ passing through $o$ with respect to the first canonical projective embedding of $M$. In particular, the subvariety $\MO_1\subset \mbP(T_o M)$ is linearly non-degenerate, i.e., not contained in a hyperplane. The $G$-orbit $\MS(\Omega)$ of a point $0\not=[\eta]\in \MO_1$ is a holomorphic bundle of homogeneous projective manifolds over $\Omega$, which is called the \emph{minimal characteristic bundle} of $\Omega$.

   Let $\Omega$ be an arbitrary bounded symmetric domain{\color{black}, and write} $\Omega=\Omega_1\times\cdots\times \Omega_m$ {\color{black} for its} decomposition into {\color{black} the Cartesian product of its} irreducible factors.  Write $T\Omega=T_1\oplus \cdots\oplus T_m$ {\color{black} for} the corresponding direct sum decomposition of the holomorphic tangent bundle. Let us denote by $\MS^i(\Omega)\subset \mbP(T{\Omega})$ the holomorphic bundle over $\Omega$ obtained from the natural embedding of $\MS(\Omega_i)\subset \mbP(T\Omega_i)$ into the projective subbundle $\mbP(T_i)\subset \mbP(T\Omega)$. Let $X=\Omega/\Gamma$ be the quotient space by a torsion-free irreducible cocompact lattice. Then the Bergman metric $ds^2_{\Omega}$ on $\Omega$ descends to a quotient metric $g$ on $X$, which is again a K{\"a}hler-Einstein metric. We have the following Finsler metric rigidity theorem on $X$ proved by the last author in \cite{mok2004extremal}.

\begin{theorem}
[\protect{\cite[Theorem and Remarks]{mok2004extremal}}]
\label{thm_Mok_rigidity}
Let $\Omega=\Omega_1\times \cdots \times \Omega_m$ be a bounded symmetric domain of rank $\geq 2$ {\color{black}together with its} decomposition into irreducible factors. Let $\Gamma\subset \Aut(\Omega)$ be a torsion-free irreducible cocompact lattice and set $X{\color{black}:}=\Omega/\Gamma$. Let $g$ be the canonical K\"ahler--Einstein metric on $X$, and let $h$ be a continuous complex Finsler pseudo-metric on $X$ such that the curvature current of the associated possibly singular continuous Hermitian metric on the line bundle $\MSO_{\mbP(TX)}(-1)$ is non-positive. {\color{black} Denote by} $\|\cdot\|_g$ (resp. $\|\cdot\|_h$) lengths of vectors measured with respect to $g$ (resp. $h$). Then there exist non-negative constants $c_1,\cdots,c_m$ such that for any $v\in TX$ that can be lifted to a vector $v'\in T_i$ with $[v']\in \MS^i(\Omega)$, $1\leq i\leq m$, we have $\|v\|_h=c_i\|v\|_g$.
\end{theorem}

\subsection{Rigidity and Integrality}

 We will apply Finsler metric rigidity to study the rigidity and integrality of irreducible compact quotients of bounded symmetric domain of complex dimension $\geq 2$. We start with the following easy lemma from linear algebra.

\begin{lemma}
    \label{l.zeros-spectral-data}
    Fix a positive integer $n$ and denote by $\sigma_1,\cdots,\sigma_n$ the elementary symmetric polynomials in $n$ variables. Let $V$ be a complex vector space of dimension $r$. Let $L_1,\cdots,L_n$ be $n$ ({\color{black} possibly} non-distinct) elements in the dual space $V^*$ and denote $\sigma_k(L_1,\cdots,L_n)$ by $P_k$, $1\leq k\leq n$. Then we have $\mbfB(L_1,\cdots,L_n)=\mbfB(P_1,\cdots,P_k)$, where
    \[
    \mbfB(L_1,\cdots,L_n)\coloneqq \{v\in V \,|\, L_k(v)=0,\,1\leq k\leq n\}
    \]
    and
    \[
    \mbfB(P_1,\cdots,P_n)\coloneqq \{v\in V\,|\,P_k(v)=0,\,1\leq k\leq n\}.
    \]
\end{lemma}

\begin{proof}
    It is clear that $\mbfB(L_1,\cdots,L_n)$ is contained in $\mbfB(P_1,\cdots,P_n)$. So it remains to show the reverse inclusion.
    
    First we claim that the set of common zeros of $\sigma_k$'s consists of only the origin. Indeed, set $\sigma_0=1$. Then we have
    \[
    \prod_{i=1}^n (X-X_i) = \sum_{i=0}^n (-1)^i \sigma_i X^{n-i}.
    \]
    In particular, if $\textbf{x}=(x_1,\cdots,x_n)$ is a common zero of $\sigma_k$, the equality above implies
    \[
    \prod_{i=1}^n(X-x_i) = X^n.
    \]
    Then letting $X=x_i$ shows that $x_i^n=0$ and hence $x_i=0$, $1\leq i\leq n$. In other words, the point $(0,\cdots,0)$ is the only common zero of $\sigma_1,\cdots,\sigma_n$.

    Next we consider the natural linear map $\Phi:V\rightarrow \mbC^n$ defined as $(L_1,\cdots,L_n)$. In particular, if we regard $\sigma_k$ as a homogeneous polynomial of degree $k$ defined over $\mbC^n$, then we get
    \[
    \mbfB(P_1,\cdots,P_k)=\mbfB(\sigma_1\circ \Phi,\cdots,\sigma_n\circ \Phi).
    \]
    We have seen from above $\mbfB(\sigma_1,\cdots,\sigma_n)=\{(0,\cdots,0)\}$. This yields
    \[
    \mbfB(P_1,\cdots,P_n)=\ker(\Phi)=\mbfB(L_1,\cdots,L_n),
    \]
    which finishes the proof.
\end{proof}

Now we are in the position to prove  Theorem \ref{t.SpBaseArithVar}.
   
\begin{proof}[Proof of Theorem \ref{t.SpBaseArithVar}]
    Assume to the contrary that $\MS^r_X\not=0$ and let $\textbf{s}=(s_1,\cdots,s_r)\in \MS^r_X$ be a non-zero element. Let $h_k$, $1\leq k\leq r$, be the induced possibly singular Hermitian metric defined by $s_k$ on the dual tautological line bundle $\MSO_{\mbP(TX)}(-1)$ over $\mbP(TX)$. {\color{black}More} precisely, for any $v\in \MO_{\mbP(TX)}(-1)$, we define the length of $v$ with respect to $h$ as following:
    \[
    \|v\|_{h_k} \coloneqq |s_k(v^k)|^{\frac{1}{k}},
    \]
    where we regard $s_k\in H^0(X,\Sym^k\Omega_X^1)$ as an element of $H^0(\mbP(TX),\MO_{\mbP(TX)}(k))$ with the canonical isomorphism
    \[
    H^0(X,\Sym^k\Omega_X^1) \cong H^0(\mbP(TX),\MO_{\mbP(TX)}(k))
    \]
    and $v^k$ is viewed as a point contained in $\MO_{\mbP(TX)}(-k)$. Then $h_k=0$ if and only if $s_k=0$, and if $s_k\not=0$, then the curvature current of $h_k$ is non-positive as shown in Example \ref{l.Current-Global-section}.

    Let $g$ be the canonical K\"ahler--Einstein metric on $X$ and {\color{black} denote by} $g'$ the induced Hermitian metric on $\MSO_{\mbP(TX)}(-1)$. Then for any $1\leq k\leq m$ such that $h_k\not=0$, by Theorem \ref{thm_Mok_rigidity}, there exist non-negative constants $c_{1k},\cdots,c_{mk}$ such that for any $v\in TX$ that can be lifted to a vector $v'\in T_i$ such that $[v']\in \MS^i(\Omega)\subset \mbP(T_i)\subset \mbP(T\Omega)$, $1\leq i\leq m$, we have
    \[
    \|v\|_{h_k,[v]} = c_{ik}\|v\|_{g',[v]}.
    \]
    The the result will follows directly from the following claim.
    \begin{claim}
    \label{c.NOn-zero}
        There exist positive integers $1\leq i\leq m$ and $1\leq k\leq r$ such that $c_{ik}>0$. 
    \end{claim}

    \begin{proof}[Proof of Claim \ref{c.NOn-zero}]
        {\color{black} We} assume to the contrary that $c_{ik}=0$ for any $i$ and $k$. Fix a point $x\in X$. By the definition of $\MS_X$, there exist $r$ (maybe non-distinct) elements $w_1,\cdots,w_r$ contained in $\Omega^1_{X,x}=T^*_x X$ such that we have
    \[
    s_k(x)=(-1)^k\sigma_k(w_1,\cdots,w_r),
    \]
    where $\sigma_k$ is the $k$-th elementary symmetric polynomial in $n$ variables. If $c_{ik}=0$, then for any element $v\in T_x X$ that can be lifted to a vector $v'\in T_i$ such that $[v']\in \MS^i(\Omega)$, we always have $s_k(v)=0$. In particular, for given $i$, as $c_{ik}=0$ for any $1\leq k\leq r$, by Lemma \ref{l.zeros-spectral-data}, it follows that for any $v\in T_x X$ that can be lifted to a vector $v'\in T_i$ such that $[v']\in \MS^i(\Omega_X)$, we have 
    \[
    w_1(v)=\cdots=w_r(v)=0.
    \]
    As a consequence, since the $w_k$'s are linear functionals over $T_x X$, it follows that $w_k$'s vanish along the linear subspace of $T_x X$ spanned by the vectors $v\in T_x X$ that can be lifted to vectors $v'\in T_i$ such that $[v']\in \MS^i(\Omega)$. On the other hand, since $\MS^i(\Omega)\subset \mbP(T_i)$ is pointwise linearly non-degenerate, it {\color{black} follows} that the $w_k$'s vanish along the subspace of $T_x X$ generated by vectors $v$ which can be lifted to $v'\in T_i$. Thus the $w_k$'s vanish identically over $T_x X$ as $i$ is arbitrary. Then $s_k(x)=0$ for all $1\leq k\leq n$ and hence $s_k=0$ for any $1\leq k\leq r$, which is absurd.
    \end{proof}

    Now choose $i$ and $k$ such that $c_{ik}\not=0$. Then for any $v\in T_x X$ that can be lifted to $v'\in T_i$ such that $[v']\in \MS^i(\Omega)$, we have $\|v\|_{h_k}=c_{ik}\|v\|_{g}$. In particular, if $v\not=0$, then we have $s_k(v)\not=0$ and hence the curvature current of $h_k$ at $v$ vanishes (see Example \ref{l.Current-Global-section}), which 
{\color{black} contradicts with Theorem 3.3.  The proof of Theorem 1.3 is complete.}
\end{proof}

\begin{remark}
    As mentioned in \S \ref{sec_introduction}, Theorem \ref{t.SpBaseArithVar} does not hold for the rank-one case, i.e., compact quotient of complex balls. However, for the Kottwitz lattice $\Gamma\subset \textup{SU}(n,1)$, B.~Klingler proved in \cite[Theorem 1.11]{Klingler2013} that $\MA_X^r=0$ for $r\leq n-1$ if $n+1$ is prime. So it should be interesting to ask whether the spectral base $\MS_X^r$ is trivial for any $r$ in this case.
\end{remark}

\section{Rigidity and integrality from spectral base perspective}
The vanishing of the Hitchin base plays a significant role in understanding the rigidity and integrality of representations, as indicated in \cite{GromovSchoen1992,arapura2002higgs,Klingler2013}. For further insights, see also \cite{JostZuo1997,cadorel2022hyperbolicity} for generalizations to the quasi-projective situation.

Moreover, it is interesting to explore the relationship between symmetric differentials and the fundamental groups, as discussed in \cite{brunebarbeklingler2013symmetric}. As a consequence of the construction by Chen--Ng\^o, we observe that the Hitchin morphism factors through the spectral base, which allows us to strengthen many statements concerning the vanishing of the Hitchin base to apply to the vanishing of the spectral base as well.

\subsection{Rigidity of the character variety}
In this subsection, we will discuss the relationship between the rigidity of the character variety and the spectral base, as outlined in \cite{arapura2002higgs}.

\begin{theorem}[\cite{arapura2002higgs}]
\label{thm_old_certierian}
If the Hitchin base $\MA^r_X= 0$, then $\mfR^{\GL_r(\mbC)}$ is rigid. 
\end{theorem}

This method has been successfully used by B.~Klingler \cite[Theorem 1.11]{Klingler2013} to study the rigidity problem of the Kottwitz lattice $\Gamma\subset {\rm SU}(n,1)$ as mentioned in the introduction. However, on the one hand it is in general not an easy task to check the vanishing of $\MA_X^r$ in Theorem \ref{thm_old_certierian} on the other hand there are many examples of varieties with rigid character variety but having non-vanishing $\MA^r_X$ -- see Examples \ref{e.BdO} and \ref{e.Shimurahigherrank} and \cite{bogomolov2011symmetric}. 
Arapura’s theorem above can be strengthened using the spectral base as follows.

\begin{theorem}
\label{thm_rigid_dolbeault_base}
The character variety $\mfR^{\GL_r(\mbC)}$ is rigid if and only if $\MS^r_{X;\Dol} = 0$. In particular, if $\MS^r_X = 0$, then $\mfR^{\GL_r(\mbC)}$ is rigid.
\end{theorem}

\begin{proof}
Note that $\mfR^{\GL_r(\mbC)}$ is an affine variety. So $\mfR^{\GL_r(\mbC)}$ is rigid if and only if it is compact. If $\MS_{X;\Dol} = 0$, then $\MMD^r \subset \ssd_X^{-1}(0)$ is compact by Theorem \ref{thm_Hitchinmap_proper} and so is $\mfR^{\GL_r(\mbC)}$ by Theorem \ref{t.Homeo}. Hence $\mfR^{\GL_r(\mbC)}$ is rigid.

Conversely, suppose that the character variety $\mfR^{\GL_r(\mbC)}$ is rigid. Then $\mfR^{\GL_r(\mbC)}$ only consists of a finite number of points. We assume to the contrary that there exists a Higgs bundle $(\msE,\vp) \in \MMD^r$ such that $\ssd_X((\msE,\vp)) \neq 0$. Note that for any $t \in \mbC^*$, the Higgs bundle $(\msE,t\vp)$ is again polystable and $\ssd_X((\msE,\vp)) \neq \ssd_X((\msE,t\vp))$ {\color{black} for $t \neq 1$}. Thus we obtain a non-trivial deformation family of Higgs bundles and then there exists a non-trivial deformation in $\mfR^{\GL_r(\mbC)}$ by the non-abelian Hodge correspondence (cf. Theorem \ref{t.Homeo}), which contradicts the assumption that $\mfR^{\GL_r(\mbC)}$ is rigid.
\end{proof}

\subsection{Harmonic Maps into Bruhat-Tits Buildings}

In this subsection, we will briefly review the construction from harmonic maps into Bruhat-Tits buildings by Gromov--Schoen \cite{GromovSchoen1992}, as well as the construction of the spectral variety by means of harmonic maps. For more details, we refer to \cite[Appendix A]{Klingler2013}, \cite{daskalopoulos2021uniqueness, cadorel2022hyperbolicity, daskalopoulos2023notes}.

Let $F$ be a non-Archimedean local field. For the group $\GL_r(F)$, one can construct a Bruhat--Tits building, denoted as $\MV$. The Bruhat--Tits building is a contractible locally finite simplicial complex. Moreover, the group $\GL_r(F)$ acts continuously on $\MV$ by simplicial automorphisms, and the action is proper. The apartments of $\MV$ are isomorphic to the Cartan subalgebra $\mfh$.

The Bruhat--Tits building is a metric space of non-positive curvature, and the theory of harmonic maps into metric spaces has been developed in \cite{GromovSchoen1992, KorevaarSchoen1993}. Let $F$ be a non-Archimedean local field. A representation $\rho:\pi_1(X)\to \GL_r(F)$ is defined to be \emph{reductive} if the Zariski closure of $\rho(\pi_1(X))$ is a reductive subgroup of $\GL_r(F)$.

Let $\tX$ be the universal cover of $X$. Given a reductive representation $\rho:\pi_1(X)\to \GL_r(F)$, by Gromov--Schoen \cite{GromovSchoen1992}, there exists a Lipschitz harmonic $\rho$-equivariant map $f:\tX\to \MV$. A point $\tx\in \tX$ is called \emph{regular} if there exists an apartment of $\MV$ containing the image by $f$ of a neighborhood of $\tx$, and other points in $\tX$ are called \emph{singular}. For the covering map $\tX\to X$, we denote by $X^{\reg}\subset X$ the image of the regular points and $X^{\sing}$ as the complement of $X^{\reg}$ in $X$. By \cite[Theorem 6.4]{GromovSchoen1992}, the subspace $X^{\sing}$ has real Hausdorff codimension at least $2$.

Let $T$ be the maximal $F$-split torus of $\GL_r(F)$ and let $\{h_1,\cdots,h_r\}$ be the root system of $T$ in $\GL_r(F)$. For each apartment $A\subset \MV$, we can take the derivative of $h_i$, obtaining $r$ real 1-forms $\{dh_1,\cdots, dh_r\}$ on the apartment $\MV$. Additionally, if $U\subset X$ is an open set with $f(U)\subset A$ consisting of only regular points, we define $\omega_i=f^{*}(dh_i)^{1,0}\in \Omega_X^1$. The harmonicity of $f$ implies that $\omega_i$ is holomorphic. Moreover, let $A_{1}$ and $A_{2}$ be two apartments. Then, over the intersection, the sets
$$
\{dh_1,\cdots,dh_r\}|_{A_{1}} \quad \text{and} \quad \{dh_1,\cdots,dh_r\}|_{A_{2}}
$$
match up to permutation by the Weyl group action. Therefore, over the regular locus $X^{\reg}$, we obtain a multivalued holomorphic 1-form.

In analogy to the definition of the Hitchin morphism, we can define a map
\begin{equation}
\label{eq_alpha_map_regular_locus}
    \begin{split}
    \alpha_{\rho}&: \bigoplus_{i=1}^r(S^i\mfh^*)^{W}\to\bigoplus_{i=1}^rH^0(X^{\reg}, \Sym^i\Omega_{X^{\reg}}^1), \\
    &(h_1,\cdots,h_r)\mapsto (\sigma_1,\cdots,\sigma_r),
    \end{split}
\end{equation}
where $\sigma_i$'s are the symmetric {\color{black} polynomials} taking values as $(h_1,\cdots,h_r)$. Based on the definition of the spectral base $\MS_{X^{\reg}}$, the image of $\alpha_{\rho}$ lies in the spectral base.

Moreover, since the singular set $X^{\sing}$ has Hausdorff codimension at least two and $f$ is a Lipschitz map, the sections $\sigma_i$ uniquely extends to $X$, and the extension of $(\sigma_1,\cdots,\sigma_r)$ also lies in the spectral base $\MS_X$ by Lemma \ref{l.birationalinvarance}. Using the same notation, we write the extension map as 
\[
\alpha_{\rho}:\bigoplus_{i=1}^r(S^i\mfh^*)^{W}\to \MA_X^r=\bigoplus_{i=1}^rH^0(X,\Sym^i\Omega_{X}^1).
\]

Let $\mfR^{\GL_r(F)}:=\{\rho|\rho:\pi_1(X)\to \GL_r(F)\}/\sim$ be the character variety. The above construction allows us to define the Hitchin morphism $\ssd_{X;F}$ for non-Archimedean representations, which is defined as
\begin{equation}
 \label{eq_Hitchin_base_for_harmonic_maps}
    \begin{split}
        \ssd_{X;F}:\mfR^{\GL_r(F)}\to \MS_X,\;\ssd_{X;F}(\rho):=\alpha_{\rho}.
    \end{split}
\end{equation}

Recall that $\rho:\pi_1(X)\to \GL_r(F)$ is said to \emph{have a bounded image} if $\rho(\pi_1(X))$ is contained in a compact subgroup of $\GL_r(F)$, where the topology of $\GL_r(F)$ is defined by the topology of the local field $F$.

\begin{definition}
Let $F$ be a non-Archimedean local field. The non-Archimedean Dolbeault spectral base $\MS_{X;\Dol}^{\NA;F}$ is defined to be $\MS_{X;\Dol}^{\NA;F}:=\ssd_{X;F}(\mfR^{\GL_r(F)})\subset \MS_X$.
\end{definition}

For any reductive representation $\rho\in \mfR^{\GL_r(F)}$, we write $f_{\rho}:\tX\to \MV$ for the corresponding harmonic map to the Bruhat-Tits building defined by $\GL_r(F)$. We now state the following theorem, which is an analogue of Theorem \ref{thm_rigid_dolbeault_base}.

\begin{theorem}
\label{thm_harmonicmap_bounded}
For a non-Archimedean local field $F$ the non-Archimedean Dolbeault spectral base with respect to $F$ vanishes, $\MS_{X;\Dol}^{\NA;F}=0$, if and only if every harmonic map $f_{\rho}$ defined by $\rho\in \mfR^{\GL_r(F)}$ is a constant map. In particular, if $\MS_X=0$, then every $\rho\in \mfR^{\GL_r(F)}$ has a bounded image.
\end{theorem}

\begin{proof} 
If $\MS_{X;\Dol}=0$, then for $\alpha_{\rho}$ in \eqref{eq_Hitchin_base_for_harmonic_maps}, we have $\alpha_{\rho}=0$, which implies that the harmonic map is a constant map. On the other side, if for every $\rho$, $f_{\rho}$ is a constant map, then based on the definition, $\alpha_{\rho}=0$. 

If $\MS_X=0$, then $f_{\rho}$ is a constant. As $\rho(\pi_1(X))$ fixes the point $f(\tX)$, and $\GL_r(F)$ acts properly on $\MV$, $\rho(\pi_1(X))$ is bounded in $\GL_r(F)$. 
\end{proof}

It would be very interesting to know the relationship between the Dolbeault spectral base $\MS_{X;\Dol}$ and the non-Archimedean Dolbeault spectral base. Given an embedding $\sigma:F\to \mbC$, then $\sigma$ induces a map $\sigma:\mfR_{\GL^r(F)}\to \mfR^{\GL_r(\mbC)}$. It will be very interesting to understand the following question:

\begin{question}
    Let $F$ be a non-Archimedean local field $F$ of characteristic zero and fix an embedding $\sigma:F\to \mbC$. For a representation $\rho\in \mfR^{\GL_r(F)}$, do we have the following equality
    \[
    \ssd_{X;F}(\rho)=\ssd_X(\xi^{-1}\circ\sigma\circ\rho)?
    \]
    Here the map $\xi^{-1}$ is the non-abelian Hodge correspondence map in \eqref{eq_NAH_map} that maps the reductive representation $\sigma\circ\rho$ to a polystable Higgs bundle.
\end{question}

\subsection{Applications}
In this subsection, we will summarize previous results in \cite{GromovSchoen1992,arapura2002higgs,Klingler2013} and rephrase them using the spectral base instead of the Hitchin base. The first one is a combination of Proposition \ref{p.CN-rankone}  with \cite[Proposition 2.4]{arapura2002higgs} and \cite[Theorem 1.6]{Klingler2013}.

\begin{theorem}
\label{thm_new_methods}
    Let $X$ be a projective manifold such that $\MS^r_{X}=0$ for some $r\geq 1$. Then the following statements hold. 
    \begin{enumerate}
        \item The character variety $\mfR^{\GL_r(\mbC)}$ is rigid.
        \item Let $F$ be any non-Archimedean field. Then any reductive representation $\rho:\pi_1(X)\to \GL_r(F)$ has bounded image.
    \end{enumerate}
\end{theorem}

\begin{proof}
(1) follows directly from Theorem \ref{thm_rigid_dolbeault_base} and (2) follows from Theorem \ref{thm_harmonicmap_bounded}.
\end{proof}

The second application is related to Simpson's integrality conjecture.

\begin{proposition}[\protect{\cite[Theorem 5]{Simpson1992} and \cite[Corollary 1.8]{Klingler2013}}]
\label{prop_bounded_Z_variation}
    Let $X$ be a projective manifold such that $\MS^r_X=0$ for some $r\geq 1$. Then any reductive representation $\rho:\pi_1(X)\to \GL_r(\mbC)$ is integral and it is a complex direct factor of a $\mbZ$-VHS.
\end{proposition}

\begin{proof}
By Theorem \ref{thm_new_methods}, the character variety $\mfR^{\GL_r(\mbC)}$ is rigid and hence is zero-dimensional. In particular, as $\mfR^{\GL_r(\mbC)}$ is defined over $\mathbb{Q}$, there exists a number field $K$ such that the point $[\rho]\in \mfR^{\GL_r(\mbC)}$ is defined over $K$. Then after replacing $\rho$ by some conjugation, we can assume that $\rho$ takes values in $\GL_r(K)$. Let $v$ be an arbitrary finite place of $K$. Then the induced representation $\rho_v:\pi_1(X)\rightarrow \GL_r(K_v)$, which is obtained from $\rho$ through the embedding $K\to K_v$, is still reductive.  So Theorem \ref{thm_new_methods} implies that $\rho_v$ has bounded image in $\GL_r(K_v)$. As $v$ is arbitrary, the image $\rho(\pi_1(X))$ lies in $\GL_r(\MO_K)$.

Finally, since $\rho(\pi_1(X))$ is contained in $\GL_r(\MO_K)$, the traces $\Tr(\rho(\gamma))$ are algebraic integers for any $\gamma \in \pi_1(X)$. So it follows from \cite[Theorem 5]{Simpson1992} and the discussion in the paragraph before \cite[Corollary 4.9]{Simpson1992} that $\rho$ is a complex direct fact of a $\mbZ$-VHS.
\end{proof}

\begin{remark}
\label{r.EsnaultTalk}
    We have {\color{black} learned} from a talk by H.~Esnault that the argument of \cite{EsnaultGroechenig2018} can be applied to show that if $\dim(\mfR^{\GL_r(\mbC)})=0$, then any representation $\rho:\pi_1(X)\rightarrow \GL_r(\mbC)$ is integral.
\end{remark}

\begin{proof}[Proof of Theorem \ref{thm_main_theorem_general_variety}]
    It follows from Theorem \ref{thm_new_methods} and Proposition \ref{prop_bounded_Z_variation}.
\end{proof}

H.~Esnault asked whether a projective manifold with infinite fundamental group must have a non-zero symmetric differential. This question was answered by Y.~Brunebarbe, B.~Klingler and B.~Totaro in \cite{brunebarbeklingler2013symmetric} for its linear version. As the last application, we obtain the following variation of \cite{brunebarbeklingler2013symmetric} following the same argument there.

\begin{theorem}[\protect{Variation of \cite[Theorem 6.1]{brunebarbeklingler2013symmetric}}]
\label{t.BKT}
    Let $X$ be a projective manifold and let $K$ be a field of characteristic zero. If there is a linear representation $\rho:\pi_1(X)\to \GL_r(K)$ such that the image is infinite, then one of the following statements hold.
    \begin{enumerate}
        \item $\MS^k_X\neq 0$ for some $k\geq 1$.

        \item The semi-simplification $\rho^{\rs}$ of $\rho$ is a complex direct factor of a $\mbZ$-VHS with infinite discrete monodromy group.
    \end{enumerate}
\end{theorem}

\begin{proof}
    Let $\rho^{\rs}:\pi_1(X)\to \GL_{r}(K)$ be the semi-simplification of $\rho$. Assume that the representation $\rho^\rs$ has finite image. Then there exists a finite \'etale covering $\pi:X'\rightarrow X$ such that $H^0(X',\Omega_{X'}^1)\not=0$ by the same argument of the proof of \cite[Theorem 6.1]{brunebarbeklingler2013symmetric}. Choose a non-zero holomorphic $1$-form $\alpha\in H^0(X',\Omega_{X'}^1)$. Then we can define a Higgs bundle on $X$ as following:
    \[
    \varphi:\MSE\coloneqq \pi_*\MO_{X'} \xrightarrow{\times \alpha} \pi_*(\MO_{X'}\otimes \Omega_{X'}^1) \xrightarrow{\cong} \pi_*\MO_{X'}\otimes \Omega_X^1 \eqqcolon \MSE\otimes \Omega_X^1,
    \]
    where the last isomorphism follows from the projection formula and the fact $\pi^*\Omega_X^1\cong \Omega_{X'}^1$ as $\pi$ is \'etale. One can easily see that the Higgs field $\varphi$ is non-zero and so $\textbf{s}:=\ssd_X([\MSE,\varphi])$ is a non-zero element in $\MS^k_X$, where $k=\deg(\pi)$. 

    Now we assume that $\MS^k_{X}=0$ for any $k\geq 1$ and the representation $\rho^{\rs}$ has infinite image. {\color{black}Then}, by Theorem \ref{thm_rigid_dolbeault_base}, the representation $\rho^{\rs}$ is rigid and then we can conclude by Proposition \ref{prop_bounded_Z_variation} {\color{black} that (2) in the statement of Theorem 4.9 holds.}
\end{proof}

\begin{remark}
    \begin{enumerate}
        \item By the Grothendieck--Riemann--Roch theorem, since $\pi:X'\rightarrow X$ is a finite \'etale morphism, the Higgs bundle $\MSE=\pi_*\MO_{X'}$ is actually topologically trivial. So the non-zero spectral datum $\textbf{s}$ is contained in $\MS_{X;\Dol}^r$,
    
        \item  We briefly recall the argument of \cite{brunebarbeklingler2013symmetric} to show the existence of non-zero symmetric forms in the second case of Theorem \ref{t.BKT}. After replacing $X$ by a finite \'etale covering $X'$, we may assume that the monodromy group $\Gamma$ is torsion-free. Let $Y$ be a resolution of the image of the period map $\Phi:X'\rightarrow D/\Gamma$. Then $Y$ is positive-dimensional as $\Gamma$ is infinite. Let $Z\rightarrow X$ be a resolution of the rational map $X'\dashrightarrow Y$. Note that the cotangent bundle $\Omega_Y^1$ is big by \cite[Corollary 3.2]{brunebarbeklingler2013symmetric}. This implies that in particular $Y$ has non-zero symmetric forms and then it also induces non-zero symmetric forms on $Z$, which
naturally descends to $X'$. As $\pi:X'\rightarrow X$ is \'etale, we have the following commutative diagram
        \[
        \begin{tikzcd}[row sep=large, column sep=large]
            \mbP(T_{X'}) \arrow[d] \arrow[r,"{\cong}"]
                & \mbP(\pi^*T_X) \arrow[d] \arrow[r,"\bar{\pi}"] 
                    & \mbP(T_X) \arrow[d]  \\
            X' \arrow[r,"="]
                & X' \arrow[r,"\pi"]
                    & X.
        \end{tikzcd}
        \]
        Recall that $\bar{\pi}^*\MO_{\mbP(T_X)}(1)\cong \MO_{\mbP(T_{X'})}(1)$ and the Iitaka dimension is preserved under finite morphisms. Hence, using the following two natural identifications for any $k\geq 0$
        \[
        H^0(\mbP(T_{X'}),\MO_{\mbP(T_{X'})}(k)) \cong H^0(X',\Sym^k \Omega_{X'}^1)
        \]
        and 
        \[
        H^0(\mbP(T_{X}),\MO_{\mbP(T_{X})}(k)) \cong H^0(X,\Sym^k \Omega_{X}^1),
        \]
        one can derive that the existence of non-zero symmetric forms on $X'$ also yields the existence of non-zero symmetric forms on $X$.

        \item One cannot expect to derive that $\MS_X^k\not=0$ for some $k\geq 1$ in the second case of Theorem \ref{t.BKT} --- see Theorem \ref{t.SpBaseArithVar} and Example \ref{e.Shimurahigherrank}.
    \end{enumerate}
\end{remark}

\begin{proof}[Proof of Corollary \ref{c.RigInte}]
    This follows directly from Theorem \ref{t.SpBaseArithVar} and Theorem \ref{thm_new_methods}.
\end{proof}

\begin{proof}[Proof of Corollary \ref{cor_nilpotent_Higgs}]
This follows directly from the vanishing of the spectral base.
\end{proof}

\subsection{Examples}
In this subsection, we summarise the relations between the various vanishing of symmetric differentials and the rigidity/integrality/finiteness of representations in the following diagram.

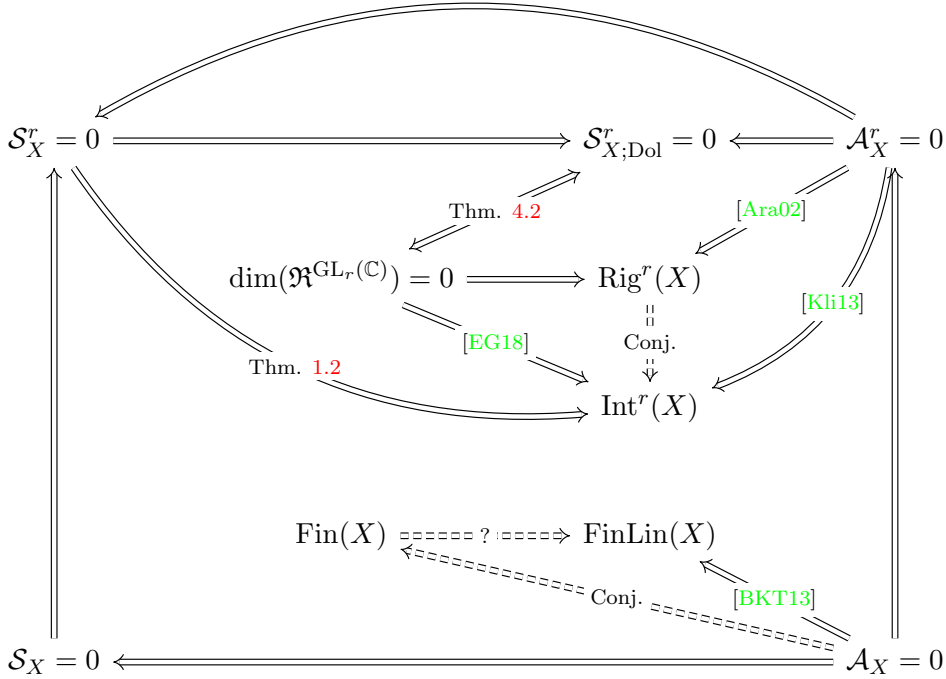
\[
\begin{tikzcd}[column sep=large, row sep=large] 
    \MS_X^r=0 \arrow[rr, Rightarrow] \arrow[ddrr, Rightarrow, bend right, "\textup{Thm. \ref{thm_main_theorem_general_variety}}" description ]
        & 
            & \MS_{X;\Dol}^r=0 \arrow[r, Leftarrow]
                & \MA_X^r=0  \arrow[dl, Rightarrow, "\textup{\cite{arapura2002higgs}}" description] \arrow[ddl, Rightarrow, bend left, "\textup{\cite{Klingler2013}}" description]  \arrow[lll, Rightarrow, bend right]\\
        & \dim(\mfR^{\GL_r(\mbC)})=0 \arrow[ur, Leftrightarrow, "\textup{Thm. \ref{thm_rigid_dolbeault_base}}" description] \arrow[r, Rightarrow] \arrow[dr, Rightarrow, "\textup{\cite{EsnaultGroechenig2018}}" description]
            & \textup{Rig}^r(X) \arrow[d, Rightarrow, dashed, "\textup{Conj.}" description]
                &  \\
        &   
            & \textup{Int}^r(X)
                &  \\
        &  \textup{Fin}(X) \arrow[r, Rightarrow, dashed, "?" description]
            & \textup{FinLin}(X)
               &   \\
    \MS_X=0 \arrow[uuuu, Rightarrow] 
        &
            &
               & \MA_X=0 \arrow[uuuu, Rightarrow] \arrow[lll, Rightarrow] \arrow[ul, Rightarrow, "\textup{\cite{brunebarbeklingler2013symmetric}}" description]
               \arrow[ull, Rightarrow, dashed, "\textup{Conj.}" description]
\end{tikzcd}
\]

Here $\textup{Rig}^r(X)$ means that every representation $\rho:\pi_1(X)\rightarrow \GL_r(\mbC)$ is rigid and $\textup{Int}^r(X)$ means that every representation $\rho:\pi_1(X)\rightarrow \GL_r(\mbC)$ is integral. The notation $\textup{Fin}(X)$ means that the fundamental group $\pi_1(X)$ is finite and $\textup{FinLin}(X)$ says that all the linear representations of $\pi_1(X)$ have finite image. Moreover, the condition $\MS_X=0$ (resp. $\MA_X=0$) just means that $\MS_X^r=0$ (resp. $\MA_X^r=0$) for all $r\geq 1$. 

\begin{example}[\protect{\cite[p.~1092, Example]{bogomolov2011symmetric}}]
\label{e.BdO}
There exists a simply connected smooth projective threefold $X$ such that $\MS_X^2\not=0$. In particular, the natural inclusion $\MS_{X;\Dol}^2\subset \MS_X^2$ is strict in this case.
\end{example}

\begin{example}
\label{e.Shimurahigherrank}
    Let $X=\Omega/\Gamma$ be a quotient of bounded symmetric domain {\color{black} of rank $\ge 2$} by an {\color{black} irreducible} torsion-free cocompact lattice as in Theorem \ref{t.SpBaseArithVar}. Then we have $\MS_X^r=0$ by Theorem \ref{t.SpBaseArithVar} and $\dim(\MA_X^r)\gg 0$ for $r$ sufficiently large. In particular, the natural inclusion $\MS_X^r\subset \MA_X^r$ is strict for $r\gg 0$. Moreover, as $\Gamma$ is infinite, it follows that $\MS_X=0$ does not imply $\textup{FinLin}(X)$; that is, the conclusion of Theorem \ref{t.BKT} cannot be improved to $\MS_X\not=0$.
\end{example}

	\bibliographystyle{alpha}
	\bibliography{references}
\end{document}

%% file: convention.tex
\newcommand{\lan}{\langle }
\newcommand{\ran}{\rangle}

\newcommand{\MS}{\mathcal{S}}
\newcommand{\MM}{\mathcal{M}}
\newcommand{\MO}{\mathcal{O}}

\newcommand{\Chow}{\mathrm{Chow}}

\newcommand{\Tr}{\mathrm{Tr}}

\newcommand{\we}{\wedge}
\newcommand{\pa}{\partial}

\newcommand{\vp}{\varphi}
\newcommand{\MA}{\mathcal{A}}

\newcommand{\Hom}{\mathrm{Hom}}

\newcommand{\mfg}{\mathfrak{g}}

\newcommand{\mfR}{\mathfrak{R}}

\newcommand{\MV}{\mathcal{V}}

\newcommand{\MMH}{\MM_{\mathrm{Higgs}}}

\newcommand{\mbP}{\mathbb{P}}

\newcommand{\Dol}{\mathrm{Dol}}

\newcommand{\MSE}{\mathscr{E}}
\newcommand{\MSF}{\mathscr{F}}

\newcommand{\MSL}{\mathscr{L}}

\newcommand{\MSO}{\mathscr{O}}

\newcommand{\MST}{\mathscr{T}}

\newcommand{\rank}{\mathrm{rank}}

\newcommand{\Aut}{\mathrm{Aut}}

\newcommand{\id}{\mathrm{id}}
\newcommand{\MMD}{\MM_{\mathrm{Dol}}}

\newcommand{\tX}{\tilde{X}}

\newcommand{\Sym}{\mathrm{Sym}}

\newcommand{\tx}{\tilde{x}}

\newcommand{\mfh}{\mathfrak{h}}

\newcommand{\mbR}{\mathbb{R}}

\newcommand{\mbZ}{\mathbb{Z}}

\newcommand{\mbfB}{\mathbf{B}}

\newcommand{\mbfs}{\mathbf{s}}

\newcommand{\mbC}{\mathbb{C}}

\newcommand{\msE}{\mathscr{E}}

\newcommand{\Higgs}{\mathrm{Higgs}}

\newcommand{\GL}{\mathrm{GL}}

\newcommand{\Tot}{\mathrm{Tot}}
\newcommand{\codim}{\mathrm{codim}}

\newcommand{\rs}{\mathrm{s}}

\newcommand\sh{\mathsf{h}}
\newcommand\ssd{\mathsf{sd}}

\newcommand{\sing}{\mathrm{sing}}
\newcommand{\reg}{\mathrm{reg}}
\newcommand{\NA}{\mathrm{NA}}

\newcommand{\mfgc}{\mfg_{\mbC}}
\newcommand{\mflc}{\mathfrak{l}_{\mbC}}
\newcommand{\mfp}{\mathfrak{p}}
\newcommand{\mfq}{\mathfrak{q}}
\newcommand{\mfl}{\mathfrak{l}}

%% file: HeLiuMok20240123_26-1-2024.bbl
\begin{thebibliography}{GRR15}

\bibitem[Ara02]{arapura2002higgs}
Donu Arapura.
\newblock Higgs bundles, integrability, and holomorphic forms.
\newblock {\em Motives, polylogarithms and Hodge theory, Part II (Irvine, CA,
  1998)}, 3:605--624, 2002.

\bibitem[BDO11]{bogomolov2011symmetric}
Fedor Bogomolov and Bruno De~Oliveira.
\newblock Symmetric differentials of rank 1 and holomorphic maps.
\newblock {\em Pure Appl. Math. Q.}, 7(4, Special Issue: In memory of Eckart
  Viehweg):1085--1103, 2011.

\bibitem[BKT13]{brunebarbeklingler2013symmetric}
Yohan Brunebarbe, Bruno Klingler, and Burt Totaro.
\newblock Symmetric differentials and the fundamental group.
\newblock {\em Duke Math. J.}, 162(14):2797--2813, 2013.

\bibitem[BKU23]{bolognese2023p}
Barbara Bolognese, Alex K{\"u}ronya, and Martin Ulirsch.
\newblock {$P=W$} phenomena on abelian varieties.
\newblock {\em arXiv preprint arXiv:2303.03734}, 2023.

\bibitem[BS94]{BandoSiu1994}
Shigetoshi Bando and Yum-Tong Siu.
\newblock Stable sheaves and {E}instein-{H}ermitian metrics.
\newblock In {\em Geometry and analysis on complex manifolds}, pages 39--50.
  World Sci. Publ., River Edge, NJ, 1994.

\bibitem[BS09]{BiswasSchumacher2009}
Indranil Biswas and Georg Schumacher.
\newblock Yang-{M}ills equation for stable {H}iggs sheaves.
\newblock {\em Internat. J. Math.}, 20(5):541--556, 2009.

\bibitem[CDY22]{cadorel2022hyperbolicity}
Benoit Cadorel, Ya~Deng, and Katsutoshi Yamanoi.
\newblock Hyperbolicity and fundamental groups of complex quasi-projective
  varieties.
\newblock {\em arXiv preprint arXiv:2212.12225}, 2022.

\bibitem[CN20]{ChenNgo2020}
Tsao-Hsien Chen and Bao~Ch\^{a}u Ng\^{o}.
\newblock On the {H}itchin morphism for higher-dimensional varieties.
\newblock {\em Duke Math. J.}, 169(10):1971--2004, 2020.

\bibitem[Cor88]{corlette1988flat}
Kevin Corlette.
\newblock Flat {$G$}-bundles with canonical metrics.
\newblock {\em J. Differential Geom.}, 28(3):361--382, 1988.

\bibitem[DM21]{daskalopoulos2021uniqueness}
Georgios Daskalopoulos and Chikako Mese.
\newblock Uniqueness of equivariant harmonic maps to symmetric spaces and
  buildings.
\newblock {\em arXiv preprint arXiv:2111.11422}, 2021.

\bibitem[DM23]{daskalopoulos2023notes}
Georgios Daskalopoulos and Chikako Mese.
\newblock Notes on harmonic maps.
\newblock {\em arXiv preprint arXiv:2301.04190}, 2023.

\bibitem[Don87]{donaldson1987twisted}
Simon~Kirwan Donaldson.
\newblock Twisted harmonic maps and the self-duality equations.
\newblock {\em Proc. London Math. Soc. (3)}, 55(1):127--131, 1987.

\bibitem[EG18]{EsnaultGroechenig2018}
H\'{e}l\`ene Esnault and Michael Groechenig.
\newblock Cohomologically rigid local systems and integrality.
\newblock {\em Selecta Math. (N.S.)}, 24(5):4279--4292, 2018.

\bibitem[Eys04]{Eyssidieux2004}
Philippe Eyssidieux.
\newblock Sur la convexit\'{e} holomorphe des rev\^{e}tements lin\'{e}aires
  r\'{e}ductifs d'une vari\'{e}t\'{e} projective alg\'{e}brique complexe.
\newblock {\em Invent. Math.}, 156(3):503--564, 2004.

\bibitem[GRR15]{garcia2015introduction}
Alberto Garc{\'\i}a-Raboso and Steven Rayan.
\newblock Introduction to nonabelian hodge theory: flat connections, higgs
  bundles and complex variations of hodge structure.
\newblock {\em Calabi-Yau Varieties: Arithmetic, Geometry and Physics: Lecture
  Notes on Concentrated Graduate Courses}, pages 131--171, 2015.

\bibitem[GS92]{GromovSchoen1992}
Mikhail Gromov and Richard~M. Schoen.
\newblock Harmonic maps into singular spaces and {$p$}-adic superrigidity for
  lattices in groups of rank one.
\newblock {\em Inst. Hautes \'{E}tudes Sci. Publ. Math.}, (76):165--246, 1992.

\bibitem[He20]{he2020behavior}
Siqi He.
\newblock The behavior of sequences of solutions to the {H}itchin-{S}impson
  equations.
\newblock {\em arXiv preprint arXiv:2002.08109}, 2020.

\bibitem[Hit87a]{hitchin1987self}
Nigel~J. Hitchin.
\newblock The self-duality equations on a {R}iemann surface.
\newblock {\em Proc. London Math. Soc. (3)}, 55(1):59--126, 1987.

\bibitem[Hit87b]{hitchin1987stable}
Nigel~J. Hitchin.
\newblock Stable bundles and integrable systems.
\newblock {\em Duke Math. J.}, 54(1):91--114, 1987.

\bibitem[Hit92]{hitchin1992lie}
Nigel~J. Hitchin.
\newblock Lie groups and {T}eichm\"{u}ller space.
\newblock {\em Topology}, 31(3):449--473, 1992.

\bibitem[HL23]{heliu2023spectralvariety}
Siqi He and Jie Liu.
\newblock On the spectral variety for rank two {H}iggs bundles.
\newblock {\em arXiv preprint arXiv:2310.18934}, 2023.

\bibitem[JZ97]{JostZuo1997}
J\"{u}rgen Jost and Kang Zuo.
\newblock Harmonic maps of infinite energy and rigidity results for
  representations of fundamental groups of quasiprojective varieties.
\newblock {\em J. Differential Geom.}, 47(3):469--503, 1997.

\bibitem[Kat97]{katzarkov1997shafarevich}
Ludmil Katzarkov.
\newblock On the {S}hafarevich maps.
\newblock In {\em Algebraic geometry---{S}anta {C}ruz 1995}, volume 62, Part 2
  of {\em Proc. Sympos. Pure Math.}, pages 173--216. Amer. Math. Soc.,
  Providence, RI, 1997.

\bibitem[Kli13]{Klingler2013}
Bruno Klingler.
\newblock Symmetric differentials, {K}\"{a}hler groups and ball quotients.
\newblock {\em Invent. Math.}, 192(2):257--286, 2013.

\bibitem[KS93]{KorevaarSchoen1993}
Nicholas~J. Korevaar and Richard~M. Schoen.
\newblock Sobolev spaces and harmonic maps for metric space targets.
\newblock {\em Comm. Anal. Geom.}, 1(3-4):561--659, 1993.

\bibitem[LZZ17]{LiZhangZhang2017}
Jiayu Li, Chuanjing Zhang, and Xi~Zhang.
\newblock Semi-stable {H}iggs sheaves and {B}ogomolov type inequality.
\newblock {\em Calc. Var. Partial Differential Equations}, 56(3):Paper No. 81,
  33, 2017.

\bibitem[Mar91]{Margulis1991}
Grigorii~Aleksandrovich Margulis.
\newblock {\em Discrete subgroups of semisimple {L}ie groups}, volume~17 of
  {\em Ergebnisse der Mathematik und ihrer Grenzgebiete (3)}.
\newblock Springer-Verlag, Berlin, 1991.

\bibitem[Mok89]{Mok1989}
Ngaiming Mok.
\newblock {\em Metric rigidity theorems on {H}ermitian locally symmetric
  manifolds}, volume~6 of {\em Series in Pure Mathematics}.
\newblock World Scientific Publishing Co., Inc., Teaneck, NJ, 1989.

\bibitem[Mok02]{Mok2002a}
Ngaiming Mok.
\newblock Characterization of certain holomorphic geodesic cycles on quotients
  of bounded symmetric domains in terms of tangent subspaces.
\newblock {\em Compositio Math.}, 132(3):289--309, 2002.

\bibitem[Mok04]{mok2004extremal}
Ngaiming Mok.
\newblock Extremal bounded holomorphic functions and an embedding theorem for
  arithmetic varieties of rank {$\geq 2$}.
\newblock {\em Invent. Math.}, 158(1):1--31, 2004.

\bibitem[OSS11]{OkonekSchneiderSpindler2011}
Christian Okonek, Michael Schneider, and Heinz Spindler.
\newblock {\em Vector bundles on complex projective spaces}.
\newblock Modern Birkh\"auser Classics. Birkh\"auser/Springer Basel AG, Basel,
  2011.
\newblock Corrected reprint of the 1988 edition, With an appendix by S. I.
  Gelfand.

\bibitem[Sch18]{schaposnik2018introduction}
L~Schaposnik.
\newblock An introduction to spectral data for higgs bundles.
\newblock {\em The geometry, topology and physics of moduli spaces of Higgs
  bundles}, pages 65--101, 2018.

\bibitem[Ser65]{Serre1965}
Jean-Pierre Serre.
\newblock {\em Alg{\`e}bre locale. {M}ultiplicit{\'e}s}, volume~11 of {\em
  Lecture Notes in Mathematics}.
\newblock Springer-Verlag, Berlin-New York, 1965.
\newblock Cours au Coll{\`e}ge de France, 1957--1958, r{\'e}dig{\'e} par Pierre
  Gabriel, Seconde {\'e}dition, 1965.

\bibitem[Sev59]{severi1959géométrie}
Francesco Severi.
\newblock La géométrie algébrique italienne. sa rigueur, ses méthodes, ses
  problèmes.
\newblock {\em Matematika}, 3(1):111--142, 1959.

\bibitem[Sim88]{Simpson1988Construction}
Carlos~T. Simpson.
\newblock Constructing variations of {H}odge structure using {Y}ang-{M}ills
  theory and applications to uniformization.
\newblock {\em J. Amer. Math. Soc.}, 1(4):867--918, 1988.

\bibitem[Sim91]{Simpson1991}
Carlos~T. Simpson.
\newblock The ubiquity of variations of {H}odge structure.
\newblock In {\em Complex geometry and {L}ie theory ({S}undance, {UT}, 1989)},
  volume~53 of {\em Proc. Sympos. Pure Math.}, pages 329--348. Amer. Math.
  Soc., Providence, RI, 1991.

\bibitem[Sim92]{Simpson1992}
Carlos~T. Simpson.
\newblock Higgs bundles and local systems.
\newblock {\em Inst. Hautes \'{E}tudes Sci. Publ. Math.}, (75):5--95, 1992.

\bibitem[Sim94a]{simpson1994moduli}
Carlos~T. Simpson.
\newblock Moduli of representations of the fundamental group of a smooth
  projective variety. {I}.
\newblock {\em Inst. Hautes \'{E}tudes Sci. Publ. Math.}, (79):47--129, 1994.

\bibitem[Sim94b]{simpson1994moduli2}
Carlos~T. Simpson.
\newblock Moduli of representations of the fundamental group of a smooth
  projective variety. {II}.
\newblock {\em Inst. Hautes \'{E}tudes Sci. Publ. Math.}, (80):5--79, 1994.

\bibitem[Sim97]{simpson1996hodge}
Carlos Simpson.
\newblock The {H}odge filtration on nonabelian cohomology.
\newblock In {\em Algebraic geometry---{S}anta {C}ruz 1995}, volume~62 of {\em
  Proc. Sympos. Pure Math.}, pages 217--281. Amer. Math. Soc., Providence, RI,
  1997.

\bibitem[SS24]{SongSun2024}
Lei Song and Hao Sun.
\newblock On the image of {H}itchin morphism for algebraic surfaces: the case
  $\textup{GL}_n$.
\newblock {\em Int. Math. Res. Not. IMRN}, (1):492--514, 2024.

\bibitem[Wen16]{Wentworth2016}
Richard~A. Wentworth.
\newblock Higgs bundles and local systems on {R}iemann surfaces.
\newblock In {\em Geometry and quantization of moduli spaces}, Adv. Courses
  Math. CRM Barcelona, pages 165--219. Birkh\"{a}user/Springer, Cham, 2016.

\bibitem[Zuo96]{zuo1996kodaira}
Kang Zuo.
\newblock Kodaira dimension and {C}hern hyperbolicity of the {S}hafarevich maps
  for representations of {$\pi_1$} of compact {K}\"{a}hler manifolds.
\newblock {\em J. Reine Angew. Math.}, 472:139--156, 1996.

\end{thebibliography}
